\documentclass[11pt,a4paper]{article}
\usepackage{amsmath,amssymb,amsthm,amscd,mathrsfs}
\usepackage{indentfirst}
\usepackage[top=25mm, bottom=30mm, left=30mm, right=30mm]{geometry}

\newtheorem{Def}{\bf Definition}[section]
\newtheorem{Thm}[Def]{\bf Theorem}
\newtheorem{Lem}[Def]{\bf Lemma}
\newtheorem{Cor}[Def]{\bf Corollary}
\newtheorem{Pro}[Def]{\bf Proposition}
\newtheorem{Rem}[Def]{\bf Remark}

\newtheorem*{claim}{\bf Claim}

\newtheorem{ThmA}{\bf Theorem}
\newtheorem{CorA}[ThmA]{\bf Corollary}

\newcommand{\R}{\mathbb{R}}
\newcommand{\C}{\mathbb{C}}
\newcommand{\Z}{\mathbb{Z}}
\newcommand{\F}{\mathbb{F}}

\newcommand{\N}{\mathbb{N}}
\newcommand{\B}{\mathbb{B}}

\newcommand{\M}{\mathbb{M}}
\newcommand{\G}{\mathbb{G}}
\newcommand{\Ad}{\operatorname{Ad}}
\newcommand{\id}{\text{\rm id}}

\newcommand{\Tr}{\mathord{\text{\rm Tr}}}

\newcommand{\ovt}{\mathbin{\overline{\otimes}}}

\newcommand{\ota}{\otimes_{\rm alg}}

\newcommand{\op}{\circ}
\title{\bf Unique prime factorization for infinite tensor product factors}
\author{Yusuke Isono\thanks{Research Institute for Mathematical Sciences, Kyoto University, 606-8502, Kyoto, Japan \protect \\  E-mail: \texttt{isono@kurims.kyoto-u.ac.jp}}}
\date{}
\begin{document}
\maketitle

\begin{abstract}
	In this article, we investigate a unique prime factorization property for infinite tensor product factors. We provide several examples of type II and III factors which satisfy this property, including all free product factors with diffuse free product components. In the type III setting, this is the first classification result for infinite tensor product non-amenable factors. 
Our proof is based on Popa's intertwining techniques and a characterization of relative amenability on the continuous cores.
\end{abstract}

\section{Introduction}

	The tensor product construction is a fundamental tool in von Neumann algebra theory. It has been used to construct interesting examples of von Neumann algebras. 
In particular, infinite tensor product factors $\ovt_{n\in \N}(M_n,\varphi_n)$, where $M_n$ are type I factors equipped with faithful normal states $\varphi_n$, attracted strong attention since it appears in the quantum field theory. We call such factors \textit{Araki--Woods factors}. All Araki--Woods factors are classified in terms of $(M_n,\varphi_n)$ \cite{Po67,AW68} and this led to the celebrated classification of \textit{amenable factors} due to Connes \cite{Co75} (see also \cite{Kr75,Ha85}).

	It is then natural to consider a classification problem of infinite tensor products constructed from \textit{non-amenable} factors. More precisely, we are interested in thinking about a classification of factors $\ovt_{n\in \N}(M_n,\varphi_n)$ in terms of non-amenable factors $(M_n,\varphi_n)$. 

	To investigate this problem, we should require some \textit{rigidity} of $M_n$. Indeed any (non-type I) infinite tensor product factor $\ovt_{n\in \N}(M_n,\varphi_n)$ is known to be \textit{McDuff}, meaning that it is stable under taking tensor products with the hyperfinite $\rm II_1$ factor $R$. Then using the decomposition $R=\ovt_{n\in \N}(R,\tau)$ where $\tau$ is the trace on $R$, one has
	$$ \ovt_{n\in \N}(M_n,\varphi_n)\simeq \ovt_{n}(M_n,\varphi_n)\ovt R \simeq \ovt_{n\in \N}(M_n\ovt R,\varphi_n \otimes \tau). $$
Observe that the tensor components determine \textit{up to} tensor products with $R$. Thus it is not easy to pick up information of $M_n$ directly. To avoid this situation, in this article, we will assume that each $M_n$ is a \textit{prime factor}, meaning that for any tensor decomposition $M_n=P\ovt Q$, we have either $P$ or $Q$ is of type I. In this case, $M_n$ is not isomorphic to $M_n \ovt R$ and we may treat $M_n$ as tensor components. We mention that in the Araki--Woods factor case, all type I factors $M_n$ are prime by definition.

	Here we briefly review the study of prime factors and related results. Examples of prime factors were first discovered by Popa \cite{Po83} and then by Ge \cite{Ge96}, in which they proved that any free group factor $L\F_{n}$ ($n \geq 2$, possibly infinite) is prime. Ozawa established a completely new and much simpler proof, using C$^*$-algebraic techniques \cite{Oz03}. 
Based on Ozawa's new proof and Popa's intertwining techniques (see Section \ref{Factors in the class P}), Ozawa and Popa obtained a remarkable structural theorem for tensor product factors \cite{OP03}. They proved that, whenever we consider a tensor product of finitely many free group factors, then the resulting tensor product factor \textit{remembers} its tensor components in the following precise sense. 
Let $M_i$ and $N_j$ be some free group factors and assume that $\ovt_{i=1}^n M_i$ and $\ovt_{j=1}^mN_j$ are stably isomorphic for some $n,m\in \N$, then $n=m$ and, after permutation of indices,  $M_i$ and $N_i$ are stably isomorphic for all $i$. Here $P$ and $Q$ are \textit{stably isomorphic} if $P\ovt \B(\ell^2)$ and $Q\ovt \B(\ell^2)$ are isomorphic.

	This should be called the \textit{unique prime factorization} for free group factors. Thus the classification of such tensor product factors is completely reduced to the one of each tensor components. This is a complete answer for the aforementioned classification problem for tensor products of finitely many free group factors.

	Many new examples of type II and III factors satisfying the unique prime factorization have been discovered since then. However, all of such results treat only tensor products of finitely many tensor components. The first example of unique prime factorization for \textit{infinite} tensor product factors are given by ourselves \cite{Is16b}, but they are all type $\rm II_1$ factors. 
	The aim of this article is to investigate the unique prime factorization for infinite tensor product factors that include type III factors. It is a more challenging problem since infinite tensor product factors $\ovt_{n\in \N}(M_n,\varphi_n)$ depend on the choice $\varphi_n$ and this dependence does not appear in both the finite tensor product case and the infinite tensor product II$_1$ factor case. 

	To introduce our main theorem, we need to prepare some notation and terminology. 
We say that an inclusion $B\subset M$ of von  Neumann algebras is \textit{with expectation} if there is a faithful normal conditional expectation from $M$ onto $B$. We say that a von Neumann algebra $M$ admits a \textit{large centralizer} if there is a finite von Neumann algebra $N\subset M$ with expectation such that $N'\cap M \subset N$. 
We define the following class of factors which satisfy a practical condition. See Section \ref{Relative amenability for subalgebras} and \ref{Factors in the class P} for the symbols $\lessdot$ and $\preceq$ respectively.

\begin{Def}\upshape\label{def of class P}
We say that a factor $M$ is in the \textit{class $\mathcal{P}$} if it is a separable non-amenable factor with a large centralizer and it satisfies the following condition:
\begin{itemize}
	\item for any separable factors $B$, $P$, $Q$ such that $B\ovt M=P\ovt Q$ and that $B$ has a large centralizer, we have either $P\preceq_{B\ovt M} B$ or $Q\lessdot_{B\ovt M} B$. 
\end{itemize}
\end{Def}
The condition in this definition may seem strange, so we briefly explain it here. Recall from \cite{Is14,Is16b} that a II$_1$ factor $M$ is \textit{strongly prime} if for any  II$_1$ factors $B,P,Q$ with $B\ovt M=P\ovt Q$, we have either $P\preceq_{B\ovt M} B$ or $Q\preceq_{B\ovt M} B$. This condition is inspired from the notion of prime numbers and has several applications including unique prime factorizations. However, in the type III setting, any of such examples are not known so far, and hence it is not appropriate to use this concept to study unique prime factorizations on type III factors. 
Now we turn to see our condition on the class $\mathcal{P}$. This condition should be regarded as a weaker version of strong primeness and, as we will see, this condition is really enough to prove our unique prime factorizations.

	We will show that any factor in the class $\mathcal{P}$ is prime, and the following factors are indeed contained in the class $\mathcal{P}$. See Section \ref{Factors in the class P} for details.
	\begin{itemize}
		\item Any non-amenable factor that satisfies condition (AO)$^+$ and has the W$^*$CBAP.
		\item Any free product factor $(M_1,\varphi_1)*(M_2,\varphi_2)$, where each $M_i$ is a diffuse von Neumann algebra with a faithful normal state $\varphi_i$. 
	\end{itemize}
We mention that in the finitely many tensor components case, the unique prime factorization of condition (AO) factors are proved in \cite{OP03,Is14,HI15} and the one of free product $\rm II_1$ factors are proved in \cite{Pe06}.

	Now we introduce the main theorem of this article.
Below we say that a factor $M$ is \textit{semiprime} if for any tensor decomposition $M=P\ovt Q$, we have either $P$ or $Q$ is amenable. We use this semiprimeness as an assumption of ambiguous factors.

\begin{ThmA}\label{ThmA}
	Let $X,Y \subset \N$ be subsets and let $M_m$ and $N_n$ be separable factors for $m\in X$ and $n\in Y$. Assume that each $M_m$ is in the class $\mathcal{P}$ and each $N_n$ is non-amenable.
If there are faithful normal states $\varphi_m$ on $M_m$, $\psi_n$ on $N_n$, amenable separable factors $M_0$ and $N_0$ (which are possibly trivial) such that
	$$M:= \ovt_{m\in X}(M_m,\varphi_m) \ovt M_0 \simeq \ovt_{n\in Y}(N_n,\psi_n) \ovt N_0,$$
then there is an injective map $\sigma \colon Y \to X$ such that $M_{\sigma(n)}\preceq_M N_n$ for all $n\in Y$. 

If we further assume that all $N_n$ are semiprime, then $\sigma$ is bijective and there are projections $p_n\in M_{\sigma(n)}$, $q_n\in N_n$ and amenable factors $R_n$ such that 
	$$p_nM_{\sigma(n)}p_n \ovt R_n \simeq q_nN_nq_n \quad \text{for all }n\in Y.$$
\end{ThmA}

By assuming that all factors belong to the class $\mathcal P$, we obtain the following unique prime factorization result. This is the first classification result for infinite tensor product type III factors in the non-amenable setting. 
We mention that, regarding free product type III factors, it is new even for finite sets $X,Y$.

\begin{CorA}\label{CorB}
	Let $X,Y \subset \N$ and let $M_m$ and $N_n$ be factors in the class $\mathcal{P}$ for all $m\in X$ and $n\in Y$. The following statements are equivalent.
\begin{itemize}
	\item There are faithful normal states $\varphi_m$ on $M_m$ and $\psi_n$ on $N_n$ and amenable factors $M_0$ and $N_0$ with separable preduals such that $\ovt_{m\in X}(M_m,\varphi_m)\ovt M_0$ and $\ovt_{n\in Y}(N_n,\psi_n)\ovt N_0$ are stably isomorphic.
	\item There is a bijection $\sigma\colon Y \to X$ such that $M_{\sigma(n)}$ and $N_n$ are stably isomorphic for all $n \in Y$.
\end{itemize}
\end{CorA}

The organization of this paper is as follows. 
In Section \ref{Preliminaries}, we recall some known facts on infinite tensor product factors and large centralizer conditions. 

In Section \ref{Relative amenability for subalgebras} (and Section \ref{Relative amenability for bimodules}), we define and study relative amenability for general von Neumann algebras as a generalization of \cite{AD93}. The main observation here is a characterization of relative amenability in terms of continuous cores (Theorem \ref{injectivity by core for relative amenable}). Using this, we prove two lemmas (Lemma \ref{intersection lemma} and \ref{infinite tensor lemma finite case}) for tensor product factors which are key ingredients of the proof of our main theorem.

In Section \ref{Factors in the class P}, we provide several examples of type II and III factors which are in the class $\mathcal{P}$. They are proved by combinations of known techniques which are first established in \cite{Oz03,IPP05}. We will use variants of them introduced in \cite{Is12a,Is16b,Io12,HU15b}.

In Section \ref{Proof of Theorem A}, we prove the main theorem. Using the condition of the class $\mathcal{P}$ and lemmas in Section \ref{Relative amenability for subalgebras}, we essentially reduce our problem to tensor product factors with finitely many tensor components. Then using techniques developed in \cite{Is14,HI15} for type III factors, we will finish the proof.

\bigskip

\noindent 
{\bf Acknowledgement.} The author would like to thank Y. Arano and R. Tomatsu for insightful discussions on infinite tensor product factors. 
He also thank the anonymous referee for careful reading of our manuscript and many insightful comments and suggestions. 
He was supported by JSPS, Research Fellow of the Japan Society for the Promotion of Science.

\tableofcontents

\bigskip

\subsection*{Notation}
Throughout the paper, we will use the following notation. Let $M$ be a von Neumann algebra and $\varphi$ a faithful normal semifinite weight on $M$. The \textit{modular operator, conjugation}, and \textit{action} are denoted by $\Delta_\varphi$, $J_\varphi$, and $\sigma^\varphi$ respectively. The \textit{continuous core} is the crossed product von Neumann algebra $M\rtimes_{\sigma^\varphi}\R$ and is denoted by $C_\varphi(M)$. The \textit{centralizer algebra} $M_\varphi$ is a fixed point algebra of the modular action. The norm $\|\, \cdot \, \|_\infty$ is the operator norm of $M$, while $\|\, \cdot \, \|_{2,\varphi}$ is the $L^2$-norm by $\varphi$. The GNS representation of $\varphi$ is denoted by $L^2(M,\varphi)$ and sometimes we omit $\varphi$ regarding as a \textit{standard representation}. See \cite{Ta01} for definitions of these objects.

For a tensor product von Neumann algebra $M\ovt N$, we always regard $M$ and $N$ as subalgebras in $M\ovt N$ via identifications $M=M\ovt \C \subset M\ovt N$, $N=\C\ovt N \subset M\ovt N$. For a von Neumann subalgebra $A\subset 1_AM1_A$ with unit $1_A$, we will write as $A\ovt N$ the von Neumann subalgebra (with unit $1_A \otimes 1_N$) generated by $a\otimes x$ for $a\in A$ and $x\in N$.

\section{Preliminaries}\label{Preliminaries}

	In this section, we recall basic properties of infinite tensor product factors. We particularly focus on the large centralizer condition of them. All results in this section should be known to experts but we could not find them in the literature. So we include all proofs for the reader's convenience. 

	We refer the reader to \cite{Co72} and \cite[Chapter XII]{Ta01} for definitions and basic facts of type $\rm III_\lambda$ factors for $0\leq \lambda \leq 1$.

\begin{Lem}\label{bicentralizer lemma}
	The following statements hold true.
\begin{itemize}
	\item[$\rm (1)$] Any semifinite factor and any type $\rm III_\lambda$ factor for $0\leq \lambda<1$ admit large centralizers.
	\item[$\rm (2)$] For any type $\rm III_1$ factor $M$ with separable predual, $M$ admits a large centralizer if and only if there is a faithful normal state $\varphi$ on $M$ such that $M_\varphi'\cap M=\C$.
\end{itemize}
\end{Lem}
\begin{proof}
	For the first statement, the finite factor case is trivial. For the semifinite and infinite case, we have only to observe that $\B(\ell^2(I))$ for any set $I$ admits an atomic masa $\ell^\infty(I)$ that is with expectation. For the type III case, Connes proved that any type $\rm III_\lambda$ factor for some $0\leq \lambda<1$ has a maximal abelian subalgebra with expectation \cite[TH\'EOR$\rm \grave{E}$M 4.2.1(a) and 5.2.1(a)]{Co72}. 

	For the second statement, if $M$ admits a large centralizer, then it has a maximal abelian subalgebra $A\subset M$ with expectation \cite[Theorem 3.3]{Po81}. Then the conclusion holds by \cite[Corollary 3.6]{HI15}.
\end{proof}

\begin{Lem}\label{states on factors}
	Let $M$ be a $\sigma$-finite factor of not type $\rm III_1$ and $\varphi$ a faithful normal state on $M$. Then for any $\varepsilon>0$, there exist a matrix unit $\{e_{i,j}\}_{i,j=1}^n$ in $M$ (possibly $n=\infty$) with the decomposition $M = eMe\ovt \B(\ell^2_n)$, where $e:=e_{1,1}$, and faithful normal states $\psi$ on $eMe$ and $\omega$ on $\B(\ell^2_n)$ such that $\|\varphi - \psi\otimes \omega\|<\varepsilon$, where $n$ and $\psi$ are taken as:
\begin{itemize}
	\item if $M$ is a type $\rm II_1$ factor, then $n<\infty$ and $\psi$ is the trace on $eMe$;

	\item if $M$ is a type $\rm II_\infty$ factor, then $n=\infty$, $e$ is a finite projection, and $\psi$ is the trace on $eMe$;

	\item if $M$ is a type $\rm III_{\lambda}$ factor for some $0<\lambda <1$, then $n<\infty$ and $\psi$ is a generalized trace in the sense that  $(eMe)'_{\psi}\cap eMe =\C$;

	\item if $M$ is a type $\rm III_0$ factor, then $n=1$, $(eMe)_{\psi}$ is of type $\rm II_1$ and $(eMe)_{\psi}' \cap eMe \subset (eMe)_{\psi}$.
\end{itemize}	
\end{Lem}
\begin{proof}
	We first study the type $\rm III_\lambda$ case. Since the T-set of $M$ is $2\pi\Z / \mathrm{log}(\lambda)$, by \cite[TH\'EOR$\rm \grave{E}$M 1.3.2]{Co72} there is a faithful normal state $\psi$ and a positive invertible operator $h \in \mathcal{Z}(M_\varphi)$ such that $\psi=\varphi(h \, \cdot \, \, )$ and $\sigma_{T}^\psi = \id$, where $T:=2\pi/\mathrm{log}(\lambda)$. By \cite[TH\'EOR$\rm \grave{E}$M 4.2.6]{Co72}, it holds that $M_\psi'\cap M = \C$ and so $M_\psi$ is a type II$_1$ factor. 
Observe $h^{-1}\in L^1(M_\psi,\psi)$ since $\psi(h^{-1}) = \varphi(1)=1 <\infty$. We can find a family of mutually equivalent and orthogonal projections $(e_i)_{i=1}^n$ from spectral projections of $h^{-1}$ such that $\|h^{-1} - \sum_{i=1}^n \mu_i e_i \|_{1,\psi} < \varepsilon$ for some $\mu_i >0$ (possibly $\mu_i=\mu_j$). Observe that 
\begin{align*}
	 \left\|\sum_{i=1}^n \psi(\mu_i e_i \, \cdot \,) - \varphi \right\| 
	= & \left\|\psi\left(\sum_{i=1}^n \mu_i e_i \, \cdot \,\right) - \psi(h^{-1}\, \cdot \,) \right\|  < \varepsilon.
\end{align*}
Let $\{e_{i,j}\}_{i,j=1}^n$ be a matrix unit in $M_\psi$ such that $e_{i,i}=e_i$ for all $i$. Then putting $\omega$ as the vector functional by $\sum_{i=1}^n \mu_i^{1/2} e_i$ (which is well defined by $\|\sum_i \mu_i^{1/2}e_i\|_{2,\psi} = \psi(\sum_i \mu_i e_i)\sim \varphi(1)=1$), one has
	$$ (\psi|_{eMe} \otimes \omega )(x) = \sum_{i=1}^n \mu_i\psi(e_{1,i}xe_{i,1}) =  \sum_{i=1}^n \mu_i\psi(e_{i,1}e_{1,i}x)=\sum_{i=1}^n \mu_i\psi(e_{i}x)  $$
for all $x\in M$ and therefore
\begin{align*}
	\|\psi|_{eMe} \otimes \omega - \varphi\| 
	=& \left\| \sum_{i=1}^n\psi(\mu_ie_i \, \cdot \, ) - \varphi\right\|  
	<  \varepsilon.
\end{align*}
Finally since $|(\psi|_{eMe}\otimes \omega)(1)-1|<\varepsilon$, up to normalizing $\psi|_{eMe}\otimes \omega$ and up to replacing $\varepsilon$ small, we obtain the desired matrix unit and states.

	The type II$_1$ factor case follows from the same argument as in the type III$_\lambda$ case, since any normal state $\varphi$ is a perturbation of the unique trace.  For the type II$_\infty$ case, we first put $\psi$ as a fixed II$_\infty$ trace, and take $h^{-1}\in L^1(M,\psi)$ as in the type III$_\lambda$ case above. One can then find a family of \textit{finite} projections $(e_i)_{i=1}^n$ which approximate $h^{-1}$ as above (here $n$ must be infinite). Then the same computations work and we get $\|\psi|_{eMe}\otimes \omega - \varphi \|<\varepsilon$. Up to normalizing, we get the conclusion.

	Finally we study the type $\rm III_0$ factor case. By \cite[TH\'EOR$\rm \grave{E}$M 5.2.1(a)]{Co72}, any faithful normal state $\psi$ on $M$ satisfies $M_{\psi}'\cap M \subset M_{\psi}$. So we only study the property that $M_\psi$ is of type $\rm II_1$. 

	By \cite[LEMME 5.2.4]{Co72}, there is a projection $e\in M_\varphi$ and an invertible positive element $h \in eM_\varphi e$ such that  $\psi_e:=\varphi(h \, \cdot \, )$ is \textit{lacunary} on $eMe$, that is, 1 is isolated in the spectrum of $\Delta_\psi$. In this case $(eMe)_{\psi_e}$ is of type $\rm II_1$ (indeed, Connes proved that $(eMe\ovt \B(\ell^2))_{\psi_e\otimes \Tr}$ is of type $\rm II_\infty$, see the last part of the proof of \cite[TH\'EOR$\rm \grave{E}$M 5.3.1]{Co72}). By replacing $h$ if necessary, we may assume that $\psi_e(e)=\varphi(e)$. 
Using Zorn's lemma, take mutually orthogonal projections $(e_i)_i$ and self-adjoint elements $h_i \in e_i M_\varphi e_i$ for each $i$ such that $\sum_{i}e_i = 1$ and each $e_i$ and $h_i$ are as above. 
Define $k:= \sum_i h_i$ as an unbounded operator affiliated with $M_\varphi$ and a faithful normal state $\psi$ on $M$ given by 
	$$\psi(x):= \varphi(k x) = \sum_i \varphi(h_i x) = \sum_i \psi_{e_i}(x), \quad x\in M^+.$$
Observe that $M_\psi$ is of type $\rm II_1$ since it contains $\sum_i (e_i M e_i)_{\psi_{e_i}}$ as a unital subalgebra. 
Since $\psi(k^{-1})=\varphi(1)=1<\infty$, $k^{-1}$ is contained in $L^1(M_\psi,\psi)$ and therefore there is a family of mutually orthogonal projections $(f_i)_{i=1}^n$ in $M_\psi$ (possibly $n=\infty$) such that $\|k^{-1} - \sum_{i=1}^n \mu_i f_i \|_{1,\psi} < \varepsilon$ for some $\mu_i >0$, so that 
	$$ \left\|\varphi - \psi\left(\sum_{i=1}^n \mu_i f_i \, \cdot \, \right)\right\|=\left\| \psi(k^{-1} \, \cdot \, ) - \psi \left(\sum_{i=1}^n \mu_i f_i \, \cdot \, \right)\right\| < \varepsilon. $$
Observe that $\psi':=\psi(\sum_i \mu_i f_i \, \cdot \, )$ is bounded since $|\psi'(1)- \varphi(1)|<\varepsilon$, and that $M_{\psi'}$ is of type $\rm II_1$ since it contains $\sum_i f_i M_\psi f_i$. Thus up to normalizing, $\psi'$ is the desired state.
\end{proof}

\subsection*{Large centralizer conditions for infinite tensor product factors}

	Until the end of this section, we fix $\sigma$-finite von Neumann algebras $M_n$ with faithful normal states $\varphi_n$ for $n\in \N$. Assume that $M_n\neq \C$ for all $n\in \N$.

\begin{Lem}\label{states of infinite tensor}
	Let $\psi_n$ be a faithful normal state on $M_n$ for all $n\in \N$. 
If $\sum_{n\in \N} \|\varphi_n-\psi_n\| < \infty$, then there is a $\ast$-isomorphism between $\ovt_{n\in \N}(M_n,\varphi_n)$ and $\ovt_{n\in \N}(M_n,\psi_n)$ which is the identity on $M_n$ for all $n\in \N$.
\end{Lem}
\begin{proof}
Let $M_0$ be the algebraic tensor product of $\{ M_n \}_n$ which is a dense $\ast$-subalgebra in both of $\ovt_{n\in \N}(M_n,\varphi_n)$ and $\ovt_{n\in \N}(M_n,\psi_n)$. Tensor product states $\varphi:=\otimes_n \varphi_n$ and $\psi:=\otimes_n\psi_n$ are well defined on $M_0$. We show that $\psi$ is well defined on $\ovt_{n\in \N}(M_n,\varphi_n)$. 

	To see this, consider faithful normal states $\omega_n$ on $\ovt_{n\in \N}(M_n,\varphi_n)$ for $n\in \N$ given by
	$$\omega_n := \psi_1 \otimes \cdots \otimes \psi_n \otimes \varphi_{n+1}\otimes \varphi_{n+2}\otimes \cdots .$$
Observe that for $n< m$,
\begin{align*}
	\|\omega_n - \omega_m\|
	\leq & \|\left( \varphi_{n+1}\otimes \cdots \otimes \varphi_{m}  \right) - \left( \psi_{n+1}\otimes \cdots \otimes \psi_{m} \right)\| \\
	\leq & \sum_{k=n+1}^{m}\| \varphi_k-\psi_k \| .
\end{align*}
So $(\omega_n)_n$ is a Cauchy sequence in the predual of $\ovt_{n\in \N}(M_n,\varphi_n)$ and converges to a normal state $\omega$. By construction, this coincides with $\psi$ on $M_0$. This means $\psi$ is well defined on $\ovt_{n\in \N}(M_n,\varphi_n)$.

	Finally applying the GNS construction for $\psi$, we have a $\ast$-homomorphism 
	$$\pi\colon \ovt_{n\in \N}(M_n,\varphi_n)\to \ovt_{n\in \N}(M_n,\psi_n)$$
which is the identity on $M_0$. By exchanging the roles of $\varphi$ and $\psi$, we get an inverse map of $\pi$ and therefore $\pi$ is a desired $\ast$-isomorphism.
\end{proof}

	The following proposition clarifies relations between infinite tensor product factors with given states $(\varphi_n)_n$ and the one with canonical states. As the proposition says, we can always choose canonical states as $(\varphi_n)_n$, \textit{up to} tensor products with Araki--Woods factors.

\begin{Pro}\label{states of infinite tensors}
	The following statements hold true.
\begin{itemize}
	\item[$\rm(1)$] If all $M_n$ are type $\rm III_1$ factors, then the infinite tensor product $\ovt_{n\in \N}(M_n,\varphi_n)$ does not depend on the choice of $\{\varphi_n\}_n$.

	\item[$\rm(2)$] Let $(\lambda_n)_n \in (0,1)^{\N}$. If each $M_n$ is a type $\rm III_{\lambda_n}$ factor for $0<\lambda_n <1$, then there are faithful normal states $\psi_n$ on $M_n$ for all $n\in \N$ such that $(M_n)'_{\psi_n}\cap M_n =\C$ and 
	$$\ovt_{n\in \N}(M_n,\varphi_n) \simeq \ovt_{n\in \N}(M_n,\psi_n) \ovt R,$$
where $R$ is an Araki--Woods factor (possibly trivial).

	\item[$\rm(3)$] If all $M_n$ are type $\rm III_{0}$ factors, then there are faithful normal states $\psi_n$ on $M_n$ for all $n\in \N$ such that $(M_n)_{\psi_n}$ is of type $\rm II_1$, $(M_n)'_{\psi_n}\cap M_n \subset (M_n)'_{\psi_n}$ and 
	$$\ovt_{n\in \N}(M_n,\varphi_n) \simeq \ovt_{n\in \N}(M_n,\psi_n) .$$

	\item[$\rm(4)$] If all $M_n$ are type $\rm II_\infty$ factors, then there are finite projections $p_n \in M_n$ for all $n\in \N$ such that 
	$$\ovt_{n\in \N}(M_n,\varphi_n) \simeq \ovt_{n\in \N}(p_nM_np_n,\tau_n) \ovt R,$$
where $\tau_n$ are traces and $R$ is an Araki--Woods factor (which must be properly infinite).

	\item[$\rm(5)$] If all $M_n$ are type $\rm II_1$ factors, then there are projections $p_n \in M_n$ for all $n\in \N$ such that
	$$\ovt_{n\in \N}(M_n,\varphi_n) \simeq \ovt_{n\in \N}(p_nM_np_n,\tau_n) \ovt R,$$
where $\tau_n$ are traces and $R$ is an Araki--Woods factor (possibly trivial). 
\end{itemize}
\end{Pro}
\begin{proof}
	Statement (1) is a straightforward consequence of Lemma \ref{states of infinite tensor} and Connes--St$\rm \o$rmer's transitivity \cite[THEOREM 4]{CS76}, while all others are of 
Lemma \ref{states of infinite tensor} and \ref{states on factors}. 
Note that for (2), we need the fact that every projections in a $\sigma$-finite type III factor are equivalent. 
\end{proof}

	We have two corollaries. We will use the first one in the proof of the main theorem.

\begin{Cor}\label{type of infinite tensor}
	If $M_n$ has a large centralizer for all $n\in \N$, then $\ovt_{n\in \N}(M_n,\varphi_n)$ has a large centralizer.
\end{Cor}
\begin{proof}
	Put $(M,\varphi):=\ovt_{n\in \N}(M_n,\varphi_n)$. Since any Araki--Woods factor has a large centralizer, by Proposition \ref{states of infinite tensors}, we may assume that $(M_n)_{\varphi_n}'\cap M_n \subset (M_n)_{\varphi_n}$ for all $n\in \N$. Then it is easy to see that $M_\varphi'\cap M\subset M_\varphi$.
\end{proof}

\begin{Cor}
	If $M_n$ has separable predual for all $n\in \N$, then $\ovt_{n\in \N}(M_n,\varphi_n)$ is McDuff (unless it is of type I).
\end{Cor}
\begin{proof}
	Since any (non-type I) Araki--Woods factor is McDuff, by Proposition \ref{states of infinite tensors}, we may assume that $(M_n)_{\varphi_n}$ is of type $\rm II_1$ for all $n\in \N$. Then the conclusion follows easily.
\end{proof}

\section{\bf Relative amenability for subalgebras}\label{Relative amenability for subalgebras}

	In this section, we define and study relative amenability for general inclusions of von Neumann algebras. 
The goal of this section is to prove two lemmas, which are necessary for our main theorem. For this, we prove a characterization of relative amenability in terms of continuous cores.
Since results in Appendix will be used, we refer the reader to the appendix section before starting this section.

	The following definition is a generalization of \cite{OP07} in which they treat only finite von Neumann algebras. 

\begin{Def}\label{relative amenable def1}\upshape
	Let $B\subset M$ be von Neumann algebras, $p\in M$ a projection and $A\subset pMp$ a von Neumann subalgebra. 
\begin{itemize}
	\item[$\rm (1)$] Let $z$ be the central support projection of  $p$ in $M$. We say that \textit{$A$ is semidiscrete relative to $B$ in $M$} if we have 
	$${}_{Mz} L^2(zMp)_A \prec {}_{Mz} L^2(zM)\otimes_{Bz} L^2(zMp)_A .$$
	\item[$\rm (2)$] Let $E_A\colon M \to A$ be a faithful normal conditional expectation. We say that the pair \textit{$(A,E_A)$ is injective relative to $B$ in $M$}  
if there exists a conditional expectation from $p\langle M,B\rangle p$ onto $A$ which restricts to $E_A$ on $pMp$. In this case we write as $(A,E_A)\lessdot_MB$.
\end{itemize}
\end{Def}

Observe that $(A,E_A)$ is injective relative to $M$ in $B$ if and only if the $pMp$-$B$-bimodule $pL^2(M)$ is left $(A,E_A)$-injective (see Appendix). Also $A$ is semidiscrete relative to $M$ in $B$ if and only if the $pMp$-$B$-bimodule $L^2(pM)$ is left $A$-semidiscrete (to see the if part, use Lemma \ref{lemma for reduction}). So Definition \ref{relative amenable def1} is a special case of Definition \ref{relative amenable def2}. 

	In item (1) above, the projection $z$ is necessary to get injectivity of the left $M$-action of the bimodule ${}_M L^2(Mp)_A$. We sometimes write this condition as, by omitting $z$, 
	$${}_{M} L^2(Mp)_A \prec {}_{M} L^2(M)\otimes_{B} L^2(Mp)_A .$$
In item (2) above, as will be explained in Remark \ref{independence of expectation}, the relative injectivity does not depend on the choice of $E_A$ if $B \subset M$ is \textit{with operator valued weight}, that is, there is a faithful normal operator valued weight from $M$ onto $B$. We refer the reader to \cite{Ha77a,Ha77b} for the theory of operator valued weights. In this case, we will simply write as $A\lessdot_MB$.

	For $M,B,p,A,E_A$ as in Definition \ref{relative amenable def1}(2) and assuming $B\subset M$ is with operator valued weight $E_B$, we will use the following notation. Let $\varphi_B$ and $\psi_A$ be faithful normal semifinite weights on $B$ and $A$ respectively, and put $\varphi:=\varphi_B\circ E_B$ and $\psi:=\psi_A\circ E_A$. We further extend $\psi$ on $M$ by adding a faithful normal semifinite weight on $(1-p)M(1-p)$, so that $\sigma^\psi_t(p)=p$ for all $t\in \R$. We have $\sigma_t^\varphi|_B=\sigma^{\varphi_B}_t$ and $\sigma_t^\psi|_A=\sigma^{\psi_A}_t$ for all $t\in \R$, and therefore there are inclusions
	$$C_\varphi(B)\subset C_\varphi(M), \quad C_\psi(A) \subset C_\psi(pMp). $$
Note that the second inclusion depends only on $\psi|_{pMp}$. Let $\Pi_{\varphi,\psi}\colon C_\psi(M)\to C_\varphi(M)$ be the canonical $\ast$-isomorphism, which is the identity on $M$.

	The following theorem establishes the equivalence of the relative injectivity of the inclusion and the one in the continuous core. Condition (2) below is particularly important to us and will be used later in this section. We note that condition (4) below is new, since it does not appear when $A=M$.

\begin{Thm}\label{injectivity by core for relative amenable}
	Let $M,B,p,A,E_A$ be as in Definition \ref{relative amenable def1}(2) and assume that $B\subset M$ is with operator valued weight $E_B$. Then using the notation introduced above, the following conditions are equivalent. 
	\begin{itemize}
		\item[$(1)$] We have $(A,E_A)\lessdot_M B$.
		\item[$(2)$] We have $C_\psi(A)\lessdot_{C_\psi(M)}B$.
		\item[$(3)$] We have $\Pi_{\varphi,\psi}\left(C_\psi(A)\right) \lessdot_{C_\varphi(M)}C_\varphi(B)$.
		\item[$(4)$] There is a ucp map $\Psi\colon p\langle M,B \rangle p \to \langle pMp,A \rangle$ such that $\Psi(x)=x$ for all $x\in pMp$. 
	\end{itemize}
\end{Thm}
\begin{proof}
	Observe first that if the central support projection $z$ of $p$ in $M$ is not 1, then all statements in this theorem are equivalent to the same statements but for the inclusions $A\subset pMzp$ and $Bz\subset Mz$. Hence up to replacing $z$ with $1_M$, without loss of generality, we may assume $z=1$.

Before starting the proof, we mention that, since there is an operator valued weight $E_B$, there is also an operator valued weight from $\langle M,B\rangle$ into $M$. This follows from \cite[Theorem 5.9]{Ha77b}. 

\noindent
	(1)$\Leftrightarrow$(2) This is exactly the equivalence of (1) and (2) in Theorem \ref{injectivity by core}, by using Lemma \ref{lemma for reduction}.

\noindent
	(2)$\Rightarrow$(3) Assuming item (2), we have $\Pi_{\varphi,\psi}(C_\psi(A))\lessdot_{C_\varphi(M)}B$. Then item (3) holds by definition, since there is a canonical inclusion $\langle C_\varphi(M), C_\varphi(B)\rangle \subset \langle C_\varphi(M), B\rangle$.

\noindent
	(3)$\Rightarrow$(4)
Since $\Pi_{\varphi,\psi}(C_\varphi(A))$ is semifinite, item (3) and Theorem \ref{relative amenable theorem2}(3) implies that there is a ucp map
	$$\Psi\colon \mathcal{L}_{C_\varphi(B)}(pL^2(C_\varphi(M))_{C_\varphi(B)}) \to \langle pC_\varphi(M)p,\Pi_{\varphi,\psi}(C_\varphi(A))$$ 
such that $\Psi(x)=x$ for all $x\in pC_\varphi(M)p$. Observe that there are identifications 
\begin{align*}
	& \mathcal{L}_{C_\varphi(B)}(pL^2(C_\varphi(M))_{C_\varphi(B)})= p\langle C_\varphi(M),C_\varphi(B)\rangle p =p\langle M,B\rangle \rtimes_\alpha \R \, p, \\
	& \langle pC_\varphi(M)p,\Pi_{\varphi,\psi}(C_\varphi(A)\rangle\simeq \langle C_\psi(pMp),C_\psi(A)\rangle = \langle pMp,A\rangle \rtimes_\beta \R,
\end{align*}
where $\alpha_t=\Ad\Delta_\varphi^{it}$ and $\beta_t=\Ad\Delta_\psi^{it}$ for $t\in \R$, and they canonical contain $p\langle M,B \rangle p$ and $\langle pMp,A \rangle$ respectively. By restriction, we have a map $\Psi\colon p\langle M,B \rangle p \to \langle pMp,A\rangle \rtimes_\beta \R$ such that $\Psi(x)=x$ for all $x\in pMp$. Finally composing this $\Psi$ with a conditional expectation from $\langle pMp,A\rangle \rtimes_\beta \R$ onto $\langle pMp,A\rangle$ and we get a desired ucp map.

\noindent
	(4)$\Rightarrow$(1)
This is trivial by composing the compression map by the Jones projection of $E_A$. 
\end{proof}

\begin{Rem}\upshape\label{independence of expectation}
	In this theorem, condition (4) does not depend on the choice of $E_A$. Hence under the assumption that $B\subset M$ is with operator valued weight, the relative injectivity does not depend the choice of $E_A$. More precisely, if $(A,E_A)\lessdot_MB$ for some $E_A$, then we have $(A,E_A')\lessdot_MB$ for any other faithful normal conditional expectation $E_A'$.
\end{Rem}

The following corollary is an immediate consequence of condition (4) above. It is a generalization of \cite[Proposition 2.4(3)]{OP07}. Our proof here is much simpler and can be applied to non tracial von Neumann algebras.

\begin{Cor}\label{transitivity}
	Let $B\subset M$ and $A\subset pMp$ be von Neumann algebras with expectations $E_A,E_B$ and let $N\subset M$ be a von Neumann subalgebra with an operator valued weight. If $(A,E_A) \lessdot_MB$ and $(B,E_B) \lessdot_MN$, then $(A,E_A) \lessdot_MN$.
\end{Cor}
\begin{proof}
	Let $\Psi\colon \langle M,N\rangle \to \langle M,B \rangle$ and $\Phi\colon p\langle M,B\rangle p \to \langle pMp,A \rangle$ be ucp maps as in Theorem \ref{injectivity by core for relative amenable}(4). Then the composition $\Phi\circ \Psi \colon p\langle M,N\rangle p \to \langle pMp,A \rangle$ works.
\end{proof}

We also prove the following useful properties.

\begin{Pro}
	Let $M,B,p,A,E_A$ be as in Definition \ref{relative amenable def1}(2) and assume that $B\subset M$ is with operator valued weight. If $(A,E_A)\lessdot_M B$, then there is a conditional expectation $E\colon p\langle M,B\rangle p\to A$ which restricts to $E_A$ on $M$ and which is approximated by normal ccp maps from $p\langle M,B\rangle p$ to $A$ in the point $\sigma$-weak topology. 
\end{Pro}
\begin{proof}
	Since the inclusion $pMp \subset p\langle M,B\rangle p $ is with operator valued weight, we can apply Corollary \ref{approximation of expectation}(2) and get the conclusion.
\end{proof}

\begin{Cor}\label{amenability for tensor product}
	For $i=1,2$, let $M_i,B_i,p_i,A_i,E_{A_i}$ be as in Definition \ref{relative amenable def1}(2) and assume that $B_i\subset M_i$ is with operator valued weight. 
Then we have that $(A_i,E_{A_i})\lessdot_{M_i} B_i$ for $i=1,2$, if and only if $(A_1\ovt A_2,E_{A_1}\otimes E_{A_2})\lessdot_{M_1\ovt M_2} B_1\ovt B_2$.
\end{Cor}
\begin{proof}
	We first assume $(A_i,E_{A_i})\lessdot_{M_i} B_i$ for $i=1,2$. By the previous proposition, for each $i$, take a net of normal ccp maps $(\varphi_{\lambda_i})_{\lambda_i}$ from $p_i\langle M_i,B_i\rangle p_i$ to $A_i$ which converges to a conditional expectation whose restriction is $E_{A_i}$ on $M_i$. 
As normal ccp maps on $p_iM_ip_i$, consider duals $\varphi_{\lambda_i}^*\colon (p_iM_ip_i)_* \to (p_iM_ip_i)_*$ and then, up to convex combinations, we may assume that $\|\varphi_{\lambda_i}^*(\omega)- (E_{A_i})^*(\omega)\|\to 0$ for all $\omega \in (p_iM_ip_i)_*$. 
Since each $\varphi_{\lambda_i}$ is normal, we can define a net of normal ccp maps $\varphi_{\lambda_1}\otimes \varphi_{\lambda_2}$ from $p_1\langle M_1,B_1\rangle p_1\ovt p_2\langle M_2,B_2\rangle p_2$ to $A_1\ovt A_2$. Let $\Phi$ be a cluster point of $\varphi_{\lambda_1}\otimes \varphi_{\lambda_2}$ in  the point $\sigma$-weak topology. 
Then an easy computation, together with the above convergence condition on $(p_iM_ip_i)_*$, implies that $\Phi|_{p_1M_1p_1\ovt p_2M_2p_2}=E_{A_1}\otimes E_{A_2}$. Hence $\Phi$ is a  conditional expectation onto $A_1\ovt A_2$ which restricts to $E_{A_1}\otimes E_{A_2}$. Finally using the identification 
	$$p_1\langle M_1,B_1\rangle p_1\ovt p_2\langle M_2,B_2\rangle p_2 = (p_1\otimes p_2 )\langle M_1 \ovt M_2 ,B_1\ovt B_2\rangle (p_1\otimes p_2 ),$$
we get the conclusion.

Conversely assume $(A_1\ovt A_2,E_{A_1}\otimes E_{A_2})\lessdot_{M_1\ovt M_2} B_1\ovt B_2$. Then using the identification above and by restriction, we have a ucp map 
	$$ \Phi \colon p_1\langle M_1,B_1\rangle p_1\otimes \C p_2\to A_1\ovt A_2$$
which restricts to $E_{A_1}\otimes E_{A_2}$ on $p_1Mp_1 \otimes \C p_2 $. Let $\varphi$ be any normal state on $A_2$ and consider a ucp map $E:=(\id_{A_1}\otimes \varphi)\circ \Phi$. Using identifications $p_1\langle M_1,B_1\rangle p_1\otimes \C p_2 = p_1\langle M_1,B_1\rangle p_1$ and $A_1\otimes \C p_2 = A_1$, $E$ is a conditional expectation from $p_1\langle M_1,B_1\rangle p_1$ onto $A_1$ which restricts to $E_1$ on $p_1M_1p_1$. 
We get $(A_1,E_{A_1})\lessdot_{M_1} B_1$ and the same argument works for $(A_2,E_{A_2})\lessdot_{M_2} B_2$.
\end{proof}

\subsection*{\bf Some lemmas for tensor product factors}

	We next prove two lemmas for tensor product factors. They are indeed key lemmas for the proof of the main theorem. We will use condition (2) of Theorem \ref{injectivity by core for relative amenable}.

	Let $X\subset \N$ and let $M_n$ be von Neumann algebras with faithful normal states $\varphi_n$ for $n\in X$. Put $(M,\varphi):=\ovt_{n\in X}(M_n,\varphi_n)$. 
For any subset $\mathcal{F}\subset X$, we write 
	$$M_{\mathcal{F}}:=\ovt_{n\in \mathcal{F}} M_n \subset M, \quad M_{\mathcal{F}}^c:=\ovt_{n\in X\setminus \mathcal{F}} M_n \subset M .$$
Observe that $M=M_{\mathcal{F}}\ovt M_{\mathcal{F}}^c$ for any $\mathcal{F} \subset X$. 
Let $p\in M$ be a projection and $P\subset pMp$ a von Neumann subalgebra with expectation $E_P$. Let $\psi$ be a faithful normal state on $M$ such that $\psi\circ E_P = \psi$ on $pMp$ and $p\in M_\psi$. Put $\widetilde{P}:=\Pi_{\varphi,\psi}(C_\psi(P))$, $\widetilde{M}:=C_\varphi(M)$ and $\widetilde{M}_{\mathcal{F}}:=C_\varphi(M_{\mathcal{F}})$. We write as $\Tr$ the canonical semifinite trace on $\widetilde{M}$. 

	The first lemma is a variant of \cite[Proposition 2.7]{PV11}. Since their proof does not work for non-finite von Neumann algebras, we prove it with a different way under a much stronger assumption.

\begin{Lem}\label{intersection lemma}
	Keep the notation and assume $X=\{1,2,3\}$. If $(P,E_P)$ is injective relative to both $Q_1:=M_1\ovt \C \ovt M_3$ and $Q_2:=\C\ovt M_2\ovt M_3$ in $M$, then $(P,E_P)$ is injective relative to $Q_1\cap Q_2 = \C \ovt \C \ovt M_3$ in $M$.
\end{Lem}
\begin{proof}
	As in the proof of Theorem \ref{injectivity by core for relative amenable}, we may assume the central support of $p$ in $M$ is 1. Then by Theorem \ref{injectivity by core for relative amenable}(2), \ref{relative amenable theorem2} and Lemma \ref{lemma for reduction}, our  assumption is equivalent to 
	$${}_{\widetilde{M}}L^2(\widetilde{M}p)_{\widetilde{P}} \prec {}_{\widetilde{M}}L^2(\widetilde{M})\otimes_{Q_i} L^2(\widetilde{M}p)_{\widetilde{P}} = {}_{\widetilde{M}} L^2(\langle \widetilde{M},Q_i\rangle p)_{\widetilde{P}}$$
for $i=1,2$. 
Using \cite[Lemma 1.7]{AD93}, we apply ${}_{\widetilde{M}} L^2(\langle \widetilde{M},Q_1\rangle )\otimes_{\widetilde{M}}$ from the left side and get that 
	$${}_{\widetilde{M}} L^2(\langle \widetilde{M},Q_1\rangle )\otimes_{\widetilde{M}} L^2(\widetilde{M}p)_{\widetilde{P}} \prec {}_{\widetilde{M}} L^2(\langle \widetilde{M},Q_1\rangle )\otimes_{\widetilde{M}} L^2(\langle \widetilde{M},Q_2\rangle p)_{\widetilde{P}}.$$
Observe that, as $\widetilde{M}$-${\widetilde{P}}$-bimodules, the left hand side satisfies
\begin{eqnarray*}
	&&{}_{\widetilde{M}} L^2(\langle \widetilde{M},Q_1\rangle )\otimes_{\widetilde{M}} L^2(\widetilde{M}p)_{\widetilde{P}} \\
	&\simeq&{}_{\widetilde{M}} \left(L^2(\widetilde{M})\otimes_{Q_1}L^2(\widetilde{M})\right)\otimes_{\widetilde{M}} L^2(\widetilde{M}p)_{\widetilde{P}} \\
	&\simeq&{}_{\widetilde{M}} L^2(\widetilde{M})\otimes_{Q_1}\left(L^2(\widetilde{M})\otimes_{\widetilde{M}} L^2(\widetilde{M}p)\right)_{\widetilde{P}} \\
	&\simeq&{}_{\widetilde{M}} L^2(\widetilde{M})\otimes_{Q_1} L^2(\widetilde{M}p)_{\widetilde{P}} \\
	&\succ& {}_{\widetilde{M}} L^2(\widetilde{M}p)_{\widetilde{P}}.
\end{eqnarray*}
Hence we obtain 
	$${}_{\widetilde{M}} L^2(\widetilde{M}p)_{\widetilde{P}} \prec {}_{\widetilde{M}} L^2(\langle \widetilde{M},Q_1\rangle )\otimes_{\widetilde{M}} L^2(\langle \widetilde{M},Q_2\rangle p)_{\widetilde{P}}.$$
Next we claim that the right hand side is actually a multiple of $L^2(\widetilde{M})\otimes_{M_3}  L^2(\widetilde{M})$. Indeed, by \cite[Proposition 2.3]{Is16a}, as $\widetilde{M}$-bimodules, we have
\begin{eqnarray*}
	&&{}_{\widetilde{M}} L^2(\langle \widetilde{M},Q_1\rangle )\otimes_{\widetilde{M}} L^2(\langle \widetilde{M},Q_2\rangle)_{\widetilde{M}} \\
	&\simeq& {}_{\widetilde{M}} \left( L^2(\widetilde{M})\otimes_{Q_1}L^2(\widetilde{M})\right)\otimes_{\widetilde{M}} \left( L^2(\widetilde{M})\otimes_{Q_2}L^2(\widetilde{M})\right)_{\widetilde{M}} \\
	&\simeq& {}_{\widetilde{M}} L^2(\widetilde{M})\otimes_{Q_1}\left[ L^2(\widetilde{M})\otimes_{\widetilde{M}} \left( L^2(\widetilde{M})\otimes_{Q_2}L^2(\widetilde{M})\right)\right]_{\widetilde{M}} \\
	&\simeq& {}_{\widetilde{M}} L^2(\widetilde{M})\otimes_{Q_1}\left[ \left(L^2(\widetilde{M})\otimes_{\widetilde{M}}  L^2(\widetilde{M})\right)\otimes_{Q_2}L^2(\widetilde{M})\right]_{\widetilde{M}} \\
	&\simeq& {}_{\widetilde{M}} L^2(\widetilde{M})\otimes_{Q_1}\left[ L^2(\widetilde{M})\otimes_{Q_2}L^2(\widetilde{M})\right]_{\widetilde{M}} \\
	&\simeq& {}_{\widetilde{M}} L^2(\widetilde{M})\otimes_{Q_1}\left[L^2(\R)\otimes L^2(M_1)\otimes L^2(Q_2)\otimes L^2(M_1)\otimes L^2(\R)\right]_{\widetilde{M}} .
\end{eqnarray*}
In the final line, we have a copy of 
	$$L^2(\langle \widetilde{M},Q_1\rangle)=L^2(\widetilde{M})\otimes_{Q_1}L^2(\R)\otimes L^2(M_1)\otimes L^2(Q_2).$$
We again apply \cite[Proposition 2.3]{Is16a} to this part and then the above bimodule is isomorphic to 
\begin{eqnarray*}
	& & {}_{\widetilde{M}} L^2(\R)\otimes L^2(M_2)\otimes L^2(Q_1)\otimes L^2(\R)\otimes L^2(M_2)\otimes L^2(M_1)\otimes L^2(\R)_{\widetilde{M}} \\
	&=& {}_{\widetilde{M}} L^2_\ell(\R)\otimes L^2_\ell(M_1)\otimes L^2_\ell(M_2)\otimes L^2_{\ell,r}(M_3)\otimes L^2(\R)\otimes L^2_r(M_2)\otimes L^2_r(M_1)\otimes L^2_r(\R)_{\widetilde{M}}. 
\end{eqnarray*}
Here we are using symbols $\ell$ and $r$ at the bottom of Hilbert spaces, which means the given left (resp.\ right) action acts on Hilbert spaces with the symbol $\ell$ (resp.\ $r$). Note that there is no actions on $L^2(\R)$, so we can ignore this part. 
We finally apply again \cite[Proposition 2.3]{Is16a} to this bimodule and then it is isomorphic to 
	$$ \bigoplus{}_{\widetilde{M}} L^2(\widetilde{M})\otimes_{M_3}  L^2(\widetilde{M})_{\widetilde{M}} ,$$
where $\bigoplus$ comes from the above $L^2(\R)$ on which there is no actions. 
Thus the claim is proven and we obtain 
	$${}_{\widetilde{M}} L^2(\widetilde{M}p)_{\widetilde{P}} \prec {}_{\widetilde{M}} L^2(\widetilde{M})\otimes_{M_3} L^2(\widetilde{M} p)_{\widetilde{P}}.$$
This exactly means $\widetilde{P}$ is semidiscrete relative to $M_3$. By Theorem \ref{injectivity by core for relative amenable}, this is equivalent to the conclusion.
\end{proof}

	The next lemma will be used to solve a problem that arises from infiniteness of tensor product components.

\begin{Lem}\label{infinite tensor lemma finite case}
	Assume that $X=\N$. If $(P,E_P)$ is injective relative to $M_{\mathcal{F}}^c$ for all {\rm finite} subsets $\mathcal{F}\subset \N \setminus \{1\}$, then $(P,E_P)$ is injective relative to $M_1$.
\end{Lem}
\begin{proof}
		As before, we may assume the central support of $p$ in $M$ is 1. Then by Theorem \ref{injectivity by core for relative amenable}(2), \ref{relative amenable theorem2} and Lemma \ref{lemma for reduction}, our assumption is equivalent to that $\widetilde{P}$ is semidiscrete relative to $M_{\mathcal{F}}^c$ in $\widetilde{M}$ for all finite subsets $\mathcal{F}\subset \N \setminus \{1\}$. We will show that $\widetilde{P}$ is semidiscrete relative to $M_1$ in $\widetilde{M}$, that is equivalent to the conclusion by Theorem \ref{injectivity by core for relative amenable}. 

	To see this, we have only to show that $r\widetilde{P}r$ is semidiscrete relative to $M_1$ for all $\Tr$-finite projections $r\in \widetilde{P}$. 
So we will indeed prove the following more general statement: let $p\in \widetilde{M}$ be a projection with $\Tr(p)<\infty$ and $P\subset p\widetilde{M}p$ be a von Neumann subalgebra. If $P$ is semidiscrete relative to $M_{\mathcal{F}}^c$ for all {\rm finite} subsets $\mathcal{F}\subset \N\setminus \{1\}$, then $P$ is semidiscrete relative to $M_1$.

	Fix a finite subset $\mathcal{F}\subset \N\setminus \{1\}$. By assumption we have a weak containment 
	$${}_{\widetilde{M}}L^2(\widetilde{M}p)_P \prec {}_{\widetilde{M}}L^2(\widetilde{M})\otimes_{M_{\mathcal{F}}^c} L^2(\widetilde{M}p)_P,$$
where we omit the support projection of $p$, as explained before. 
Here we claim that $L^2(\widetilde{M})\otimes_{M_{\mathcal{F}}^c} L^2(\widetilde{M})$ is, as a $\widetilde{M}_{\mathcal{F}\cup\{1\}}$-$\widetilde{M}$-module, a multiple of $L^2(\widetilde{M})\otimes_{M_1} L^2(\widetilde{M})$, 
so that we indeed obtain 
	$${}_{\widetilde{M}_{\mathcal{F}\cup\{1\}}}L^2(\widetilde{M}p)_P \prec {}_{\widetilde{M}_{\mathcal{F}\cup\{1\}}}L^2(\widetilde{M})\otimes_{M_1} L^2(\widetilde{M}p)_P.$$

	We prove the claim. Since $M=M_{\mathcal{F}}\ovt M_{\mathcal{F}}^c$, by \cite[Proposition 2.3]{Is16a}, we have a canonical $\widetilde{M}$-bimodule isomorphism 
\begin{align*}
	L^2(\widetilde{M})\otimes_{M_{\mathcal{F}}^c} L^2(\widetilde{M}) 
	=&\ L^2(\R)\otimes L^2(M_\mathcal{F}) \otimes L^2(M_{\mathcal{F}}^c)\otimes L^2(M_{\mathcal{F}})\otimes L^2(\R)\\
	=&\  L^2(\R)\otimes L^2(M_\mathcal{F}) \otimes L^2(M_1)\otimes L^2(M_{\mathcal{F}\cup\{1\}}^c)\otimes L^2(M_{\mathcal{F}})\otimes L^2(\R).
\end{align*}
If we write as $L^2_\ell(\R)\otimes H_1\otimes H_2 \otimes H_3\otimes H_4\otimes L^2_r(\R)$ the Hilbert space in the final line, then the left $\widetilde{M}_{\mathcal{F}\cup \{1\}}$-action is the one on $L^2_\ell(\R) \otimes H_1\otimes H_2$ and the right-one is on $H_2\otimes H_3\otimes H_4\otimes  L^2_r(\R)$. We also consider by \cite[Proposition 2.3]{Is16a}
\begin{align*}
	&L^2(\widetilde{M})\otimes_{M_1} L^2(\widetilde{M}) \\
	=& \ L^2(\R)\otimes L^2(M_\mathcal{F}) \otimes L^2(M_{\mathcal{F}\cup\{1\}}^c) \otimes L^2(M_1)\otimes L^2(M_{\mathcal{F}\cup\{1\}}^c)\otimes L^2(M_{\mathcal{F}})\otimes L^2(\R)
\end{align*}
and observe that the difference of $L^2(\widetilde{M})\otimes_{M_{\mathcal{F}^c}} L^2(\widetilde{M})$ and $L^2(\widetilde{M})\otimes_{M_1} L^2(\widetilde{M})$ is only the component $L^2(M_{\mathcal{F}\cup \{1\}}^c)$, on which there is no left-right actions as $\widetilde{M}_{\mathcal{F}\cup\{1\}}$-$\widetilde{M}$-module. Thus we obtain the desired result, and the claim is proven.

	Since the resulting weak containment holds for all finite subsets $\mathcal{F}\subset \N\setminus \{1\}$, if we put 
	$$\widetilde{M}_{\rm fin}:= \text{the norm closure of }\bigcup_{\mathcal{F}\subset \N, \text{ finite}}\widetilde{M}_{\mathcal{F}} \subset \widetilde{M},$$
which is a C$^*$-algebra, then we have 
	$${}_{\widetilde{M}_{\rm fin}}L^2(\widetilde{M}p)_P \prec {}_{\widetilde{M}_{\rm fin}}L^2(\widetilde{M})\otimes_{M_1} L^2(\widetilde{M}p)_P.$$
Let $\pi$ denote the left $\widetilde{M}$-action and $\theta$ the right $\widetilde{P}$-action on ${}_{\widetilde{M}}L^2(\widetilde{M})\otimes_{M_1} L^2(\widetilde{M}p)_P$. 
Let $\nu$ be the algebraic $\ast$-homomorphism corresponding to the above weak containment. We define an algebraic positive linear functional 
	$$\Omega \colon \text{$\ast$-alg}\{ \pi(\widetilde{M}), \theta(P^{\rm op})\} \to \C ; \ \Omega(a\otimes_{M_1} b^{\rm op}):= \Tr(ab), \quad a\in \widetilde{M}, \ b\in P.$$
This is indeed a positive linear functional, since it is a composition of $\nu$ and the vector state by $p\in L^2(\widetilde{M})$. We know that $\Omega$ is bounded on $\ast$-alg$\{ \pi(\widetilde{M}_{\rm fin}), \theta(P^{\rm op})\}$. 
We claim that $\Omega$ is bounded on the whole domain. By Lemma \ref{GNS lemma} this is equivalent to that $P$ is semidiscrete relative to $M_1$ in $\widetilde{M}$, which is our conclusion.

	We prove the claim. For any subset $\mathcal{F}\subset \N$, let $E_\mathcal{F}$ denotes the canonical conditional expectation from $\widetilde{M}$ onto $\widetilde{M}_{\mathcal{F}}$. Observe that $\id_{\widetilde{M}}=\lim_{\mathcal{F}} E_{\mathcal{F}}$ in the point strong topology, where the limit is taken over all \textit{finite} subsets of $\N$. 
Using the fact $E_{\mathcal{F}}(\widetilde{M}) \subset \widetilde{M}_{\rm fin}$ and writing as $C>0$ the bound of $\Omega$ on the dense domain, we compute that for any $a_i\in \widetilde{M}$, $b_i\in P$,
\begin{eqnarray*}
	\left|\sum_{i=1}^n\Tr( a_ib_i)\right|
	&=& \lim_{\mathcal{F}}\left|\sum_{i=1}^n\Tr( E_{\mathcal{F}}(a_i)b_i)\right| \\
	&=& \lim_{\mathcal{F}}\left|\Omega\left(\sum_{i=1}^n E_{\mathcal{F}}(a_i)\otimes b_i^{\rm op}\right)\right| \\
	&\leq& \lim_{\mathcal{F}}C \cdot \left\|\sum_{i=1}^n E_{\mathcal{F}}(a_i)\otimes b_i^{\rm op}\right\|_{\infty} \\
	&\leq& \lim_{\mathcal{F}}C \cdot \|E_{\mathcal{F}}\otimes \id_{P^{\rm op}}\|\left\|\sum_{i=1}^n a_i\otimes b_i^{\rm op}\right\|_{\infty} \\
	&=& C \cdot \left\|\sum_{i=1}^n a_i\otimes b_i^{\rm op}\right\|_{\infty}. 
\end{eqnarray*}
Thus we obtain the boundedness of the desired map.
\end{proof}

\section{\bf Factors in the class $\mathcal{P}$}\label{Factors in the class P}

\subsection*{Popa's intertwining techniques}

We recall Popa's intertwining techniques \cite{Po01,Po03}. We introduce the one formulated in \cite{HI15} for general $\sigma$-finite von Neumann algebras.

\begin{Def}\label{definition intertwining}\upshape
Let $M$ be any $\sigma$-finite von Neumann algebra, $1_A$ and $1_B$ any nonzero projections in $M$, $A\subset 1_AM1_A$ and $B\subset 1_BM1_B$ any von Neumann subalgebras with faithful normal conditional expectations $E_A : 1_A M 1_A \to A$ and $E_B : 1_B M 1_B \to B$ respectively.  

We will say that $A$ {\em embeds with expectation into} $B$ {\em inside} $M$ and write $A \preceq_M B$ if there exist projections $e \in A$ and $f \in B$, a nonzero partial isometry $v \in fMe$ and a unital normal $\ast$-homomorphism $\theta : eAe \to fBf$ such that the inclusion $\theta(eAe) \subset fBf$ is with expectation and $va =  \theta(a)v$ for all $a \in eAe$.
\end{Def}

	We prove some lemmas.

\begin{Lem}[{\cite[Lemma 4.9]{HI15}}]\label{relative commutant}
	Keep the notation as in the previous definition. If $A\preceq_M B$, then $B'\cap 1_BM1_B \preceq_M A'\cap 1_A M 1_A$.
\end{Lem}

\begin{Lem}\label{tensor lemma}
	Keep the notation as in the previous definition and let $N_0\subset N$ be any inclusion of $\sigma$-finite von Neumann algebras with expectation $E_{N_0}$. Then $A\preceq_M B$ if and only if $A\ovt N_0 \preceq_{M\ovt N} B\ovt N$.
\end{Lem}
\begin{proof}
	The case $A$ finite is proved in \cite[Lemma 4.6]{HI15}. Assume that $A\preceq_M B$ and take $p,q,\theta,v$ as in the definition. Then $p\otimes1, q\otimes 1, \theta\otimes \id, v\otimes 1$ work for $A\ovt N_0 \preceq_{M\ovt N} B\ovt N$. 

Assume next that $A\ovt N_0 \preceq_{M\ovt N} B\ovt N$. By \cite[Theorem 2(ii)]{BH16}, take a nonzero positive element $d\in (A\ovt N_0)' \cap 1_A \langle M\ovt N,\widetilde{B}\ovt N \rangle 1_A$ such that $d1_AJ1_BJ =d$ and $T(d)\in M\ovt N$, where $J$ is the modular conjugation for $L^2(M\ovt N)$, $\widetilde{B}$ is the unitization of $B$ in $M$, and $T$ is the operator valued weight from $\langle M\ovt N,\widetilde{B}\ovt N \rangle$ to $M\ovt N$ corresponding to $E_{\widetilde{B}} \ovt \id_{N}$. 
Let $\psi$ be a faithful normal state on $N$ such that $\psi\circ E_{N_0}=\psi$. Observe that
	$$(A\ovt N_0)' \cap  1_A\langle M\ovt N ,B \ovt N \rangle 1_A = \left(A'\cap 1_A\langle M,B \rangle 1_A \right)\ovt \left(N_0' \cap  N\right)$$
and hence $d_0:=(\id\otimes \psi)(d)$ is a nonzero positive element in $A'\cap 1_A\langle M,B \rangle 1_A$ satisfying $d_01_AJ1_BJ =d_0$. Observe that $\langle M\ovt N,\widetilde{B}\ovt N \rangle = \langle M ,\widetilde{B} \rangle \ovt N$ and $T$ is of the form $T_0 \otimes \id_N$, where $T_0$ is the operator valued weight corresponding to $E_{\widetilde{B}}$. Hence we have 
	$$ T_0(d_0) = (T\otimes \psi) (d) = (\id\otimes \psi) (T(d)) \in M.$$
By \cite[Theorem 2(ii)]{BH16}, we obtain $A\preceq_M B$.
\end{proof}

\begin{Lem}\label{uniqueness lemma}
	Let $M,N$ be $\sigma$-finite von Neumann algebras, $p\in M$ a projection and $A \subset pMp$ a diffuse von Neumann subalgebra with expectation. Then we have $A \not\preceq_{M\ovt N} N$.
\end{Lem}
\begin{proof}
	This is actually proved in the last part of the proof of \cite[Theorem 5.6]{HI15}. Since $A$ is diffuse, there is a diffuse  abelian von Neumann subalgebra $A_0\subset A$ with expectation. Using \cite[Lemma 4.8]{HI15}, up to replacing $A$ with $A_0$, we may assume $A$ is abelian. 
Let $(u_n)_n$ be a sequence of unitaries in $A$ such that $u_n\to 0$ weakly. Then a simple computation yields that $ E_{N}( (a\otimes b)^* u_n (c\otimes d) ) \to 0 $ strongly for all $a,c\in M$ and $b,d\in N$, where $E_N$ is a faithful normal conditional expectation given by $E_N = \varphi\otimes \id_N$ for a faithful normal state $\varphi$ on $M$. This implies the conclusion by \cite[Theorem 4.3(5)]{HI15}.
\end{proof}

\subsection*{Factors in the class $\mathcal{P}$}

We show examples mentioned in Introduction are indeed contained in the class $\mathcal P$. For this we prepare a few lemmas.

\begin{Lem}\label{stably isom for class P}
	Let $M,N$ be separable factors, $p\in M$, $q\in N$ projections, and let $I_1$, $I_2$ be type $\rm I$ separable factors. If $pMp\ovt I_1 \simeq qNq \ovt I_2$ and if $M$ is in the class $\mathcal P$, then $N$ is in the class $\mathcal P$.
\end{Lem}
\begin{proof}
Let $I_\infty$ be the type I$_\infty$ factor. Put $I_\infty^i := I_i \ovt I_\infty$ for $i=1,2$ and observe that they are of type I$_\infty$ and therefore properly infinite. We get $M\ovt I^1_\infty = N \ovt I^2_\infty$. 
Let $N\ovt B= P\ovt Q$ be as in Definition \ref{def of class P} and we will show $P\preceq_{N\ovt B}B$ or $Q\lessdot_{N\ovt B}B$. By tensoring with $I^2_\infty$, we have 
	$$ I^2_\infty\ovt P\ovt Q =  I^2_\infty \ovt N \ovt B= M\ovt I^1_\infty \ovt B.$$
Since $M$ is in the class $\mathcal{P}$, we have either 
	$$\text{(i) }I^2_\infty\ovt P \preceq_{M \ovt I^1_\infty \ovt B} B\ovt I^1_\infty; \quad \text{or} \quad \text{(ii) } Q \lessdot_{M\ovt I^1_\infty \ovt B} I_\infty^1 \ovt B.$$ 

	Assume (i). Let $e_i \in I_\infty^i$ be minimal projections for $i=1,2$. By \cite[Remark 4.2(4) and 4.5]{HI15}, we have $\C e_2\ovt P \preceq_{M \ovt I^1_\infty \ovt B} B\ovt \C e_1$, and hence $ P \preceq_{M \ovt I^1_\infty \ovt B} B$ by \cite[Remark 4.2(2)]{HI15}. 
Using the isomorphism $M \ovt I^1_\infty \ovt B=N\ovt I_\infty^2 \ovt B$ and applying Lemma \ref{tensor lemma}, we can remove $I_\infty^2$ and obtain $P \preceq_{N  \ovt B} B$.

	Assume (ii). Since $I_\infty^1$ is amenable, it holds that $I^1_\infty \lessdot_{I^1_\infty} \C $. Combined with a trivial condition $B \lessdot_{M\ovt B}B$ and using Corollary \ref{amenability for tensor product}, we get $I_\infty^1 \ovt B \lessdot_{M\ovt I_\infty^1\ovt B} \C\ovt B$. 
The assumption (ii) and Corollary \ref{transitivity} then implies $Q \lessdot_{M \ovt I^1_\infty \ovt B} B$. Using $M \ovt I^1_\infty \ovt B=N \ovt I^2_\infty \ovt B$ and applying Corollary \ref{amenability for tensor product}, we get that $Q\lessdot_{N\ovt B}B$. 
\end{proof}

\begin{Lem}\label{primeness of the class P}
	Let $M$ be a factor in the class $\mathcal P$. Then it is prime.
\end{Lem}
\begin{proof}
	Suppose by contradiction that $M$ has a decomposition $M=M_1\ovt M_2$ for diffuse factors $M_1$ and $M_2$. Since $M$ is non-amenable, we may assume $M_2$ is non-amenable. 
Let $B:=\C$ and consider 
	$$ B\ovt M = (B\ovt M_1) \ovt M_2 = P\ovt Q,$$
where $P:=B\ovt M_1$ and $Q:=M_2$. By the definition of the class $\mathcal P$, we have either $P\preceq_{B\ovt M} B$ or $Q\lessdot_{B\ovt M} B$. 

The first condition means $ M_1 \preceq_M \C$ which contradicts the diffuseness of $M_1$. The second condition means that $M_2 \lessdot_M \C$ which contradicts non-amenability of $M_2$. 
Thus in each case, we get a contradiction.
\end{proof}

\begin{Lem}\label{lemma for class P}
	Let $M$ be a separable non-amenable factor having a large centralizer. Then $M$ is in the class $\mathcal{P}$ if and only if it satisfies the condition in Definition \ref{def of class P} by assuming that $B,P,Q$ are type $\rm III_1$ factors having large centralizers.
\end{Lem}
\begin{proof}
	We show the `if' direction. Let $B,P,Q$ be as in Definition \ref{def of class P} and assume that $P \not\preceq_{B\ovt M} B$. We will show that $Q \lessdot_{B\ovt M} B$. 
Let $R_\infty$ be the Araki--Woods factor of type $\rm III_1$ and decompose it as $R_\infty = R_1\ovt R_2$, where $R_1\simeq R_2\simeq R_\infty$. Consider 
	$$\widetilde{B} \ovt M = \widetilde{P} \ovt \widetilde{Q} , \quad\text{where } \widetilde{B}:=R_\infty \ovt B, \ \widetilde{P}:=R_1 \ovt P, \ \widetilde{Q}:=R_2 \ovt Q. $$
The assumption $P \not\preceq_{B\ovt M} B$ is equivalent to $\widetilde{P} \not\preceq_{\widetilde{B}\ovt M} \widetilde{B}$ by Lemma \ref{tensor lemma}. 
Observe that by \cite[Theorem G]{AHHM18}, $\widetilde{P}$ and $\widetilde{Q}$ must admit large centralizers (since so does $P\ovt Q$ by assumption). 
Hence if $M$ satisfies the `if' condition of the statement, since $\widetilde{B},\widetilde{P},\widetilde{Q}$ are type $\rm III_1$ factors with large centralizers, we get that $\widetilde{Q} \lessdot_{\widetilde{B}\ovt M} \widetilde{B}$. 
By Corollary \ref{amenability for tensor product}, this implies $Q \lessdot_{B\ovt M} B$ and this is the desired condition.
\end{proof}

\begin{Lem}\label{condition for class P}
	Let $M$ be a separable non-amenable factor having a large centralizer. Assume $M$ satisfies the following condition:
\begin{itemize}
	\item for any separable type $\rm III_1$ factor $B$ and an abelian von Neumann subalgebra $A\subset B\ovt M$ with expectation, we have either $A\preceq_{B\ovt M}B$ or $A'\cap (B\ovt M) \lessdot_{B\ovt M} B$.
\end{itemize}
Then $M$ is in the class $\mathcal{P}$.
\end{Lem}
\begin{proof}
	Let $B,P,Q$ be as in Definition \ref{def of class P} and assume that $P \not\preceq_{B\ovt M} B$. We will show that $Q \lessdot_{B\ovt M} B$. Thanks to Lemma \ref{lemma for class P}, we may assume that $B,P,Q$ are type $\rm III_1$ factors having large centralizers. 
Since $P$ has a large centralizer and is of type $\rm III_1$, by Lemma \ref{bicentralizer lemma} it has a type $\rm II_1$ subfactor $N\subset P$ with expectation such that $N'\cap P=\C$. 
Observe that we have $N \not\preceq_{B\ovt M} B$ by Lemma \ref{relative commutant} (indeed $N \preceq_{B\ovt M} B$ implies $P\preceq_{B\ovt M} B$ by taking relative commutant two times). Using \cite[Corollary 4.7]{HI15}, there is an abelian von Neumann subalgebra $A\subset N$ with expectation such that $A \not\preceq_{B\ovt M} B$. 
Now we apply the assumption of $M$ in the statement and get that $A'\cap M \lessdot_{B\ovt M} B$. Since $Q \subset A'\cap M$ is with expectation, we conclude that $Q \lessdot_{B\ovt M} B$.
\end{proof}

\begin{Thm}
	The following factor $M$ belongs to the class $\mathcal{P}$.
\begin{itemize}
	\item[$\rm (i)$] A free product von Neumann algebra $M:=(M_1,\varphi_1)*(M_2,\varphi_2)$, where $(M_i,\varphi_i)$ are diffuse von Neumann algebras with separable predual equipped with faithful normal states. 

	\item[$\rm (ii)$] A non-amenable separable factor $M$ that satisfies condition (AO)$^+$ in the sense of \cite[Definition 3.1.1]{Is12a} and has the W$^*$CBAP (e.g.\ \cite[\S12.3]{BO08}). 
This includes the following examples (see also \cite[Remarks 2.7(3)]{HI15}):
\begin{itemize}
	\item any group von Neumann algebra $L\Gamma$, where $\Gamma$ is an ICC, non-amenable and weakly amenable discrete group which is bi-exact in the sense of \cite[\S15.1]{BO08};
	\item any compact quantum group von Neumann algebra $L^\infty(\G)$ that is a non-amenable factor, where $\widehat{\G}$ is weakly amenable and bi-exact (see \cite[Theorem C]{Is13});
	\item any free Araki--Woods factor \cite{HR10}\cite[Appendix C]{HI15}.
\end{itemize}
\end{itemize}
\end{Thm}
\begin{proof}
	The second statement follows from Ozawa's celebrated solidity theorem \cite{Oz03}. Indeed the large centralizer condition is verified in \cite[Theorem 3.7]{HI15} by the solidity. 
Proceeding as in the proof of \cite[Proposition 7.3]{Is16b} (which is a generalization of \cite{Oz03}), we can prove the condition in Lemma \ref{condition for class P} and therefore $M$ is in the class $\mathcal{P}$. See also the proof of \cite[Theorem 5.3.3]{Is12b} which treats Ozawa's proof for type III factors.

	We see the first statement. The factoriality and the large centralizer condition are proved in \cite[Theorem 3.4]{Ue10} and \cite[Theorem A.1]{HU15a} respectively. It is full by \cite[Theorem 3.7]{Ue10}, so it is non-amenable. 
So we will check only the condition in Definition \ref{def of class P}.

	Let $M$ be the free product as in the statement and let $B,P,Q$ be as in Definition \ref{def of class P}. By Lemma \ref{lemma for class P}, we assume that $B,P,Q$ are type $\rm III_1$ factors having large centralizers. As in the proof of Lemma \ref{condition for class P}, we can find type $\rm II_1$ subfactors $P_0\subset P$ and $Q_0\subset Q$ with expectations and with trivial relative commutants, and an abelian subalgebra $A\subset P_0$ such that $A\not \preceq_{B\ovt M} B$. We will show that $Q \lessdot_{B\ovt M}B$. 

	Let $\varphi_M$ be the free product state on $M$ and $\varphi_B$, $\psi_P$, $\psi_Q$ faithful normal states on $B$, $P$, $Q$ respectively. We may assume $P_0=P_{\psi_P}$ and $Q_0=Q_{\psi_Q}$. We put $N:=B\ovt M$,  $\varphi:=\varphi_B\otimes \varphi_M$, $\psi:=\psi_P\otimes \psi_Q$, and consider continuous cores $\widetilde{N}:=C_\varphi(N)$, $\widetilde{B}:=C_\varphi(B)$, $\widetilde{Q}:=\Pi_{\varphi,\psi}(C_\psi(Q))$, $\widetilde{P}:=\Pi_{\varphi,\psi}(C_\psi(P))$ and $\widetilde{A}:=\Pi_{\varphi,\psi}(C_\psi(A))$. We write as $\Tr$ the canonical trace on $\widetilde{N}$. 
Observe that $\widetilde{A}$ is abelian containing $A$ and the inclusion $A\subset \widetilde{N}$ is with expectation. 
For any $\Tr$-finite projection $e\in \widetilde{A}$, we have $Ae \not\preceq_{\widetilde{N}}\widetilde{B}$ by \cite[Proposition 2.10]{BHR12}. Observe that there is the amalgamated free product structure 
	$$\widetilde{N}=C_{\varphi}(B\ovt M_1)*_{\widetilde{B}} C_\varphi(B\ovt M_2).$$ 
We apply \cite[Theorem A.4]{HU15b} and get the following result: for any $\Tr$-finite projection $e\in \widetilde{A}$, we have either one of the following conditions:
	$$\text{(i) } \mathcal{N}_{e\widetilde{N}e}(Ae)''\lessdot_{\widetilde{N}}\widetilde{B}; \quad \text{or} \quad \text{(ii) } \mathcal{N}_{e\widetilde{N}e}(Ae)''\preceq_{\widetilde{N}}C_\varphi(B\ovt M_i) \text{ for some }i\in \{1,2\}.$$

\noindent
\textbf{Suppose first that (i)} $\mathcal{N}_{e\widetilde{N}e}(Ae)''\lessdot_{\widetilde{N}}\widetilde{B}$ for all such $e$. Observe that $e\widetilde{A}^ce\subset \mathcal{N}_{e\widetilde{N}e}(Ae)''$, where $\widetilde{A}^c:=\Pi_{\varphi,\psi}(C_\psi(A'\cap N)) $. We have $\widetilde{A}^c\lessdot_{\widetilde{N}}\widetilde{B}$ which implies that ${A}^c\lessdot_{{N}}{B}$ by Theorem \ref{injectivity by core for relative amenable}. Hence we obtain ${Q}\lessdot_{{N}}{B}$ and get the conclusion. 

\vspace{0.5em}
\noindent
\textbf{Suppose next that (ii)} $\mathcal{N}_{e\widetilde{N}e}(Ae)''\preceq_{\widetilde{N}}C_\varphi(B\ovt M_i)$ for some $i\in \{1,2\}$ and for a projection $e$. We have $\widetilde{A}^ce\preceq_{\widetilde{N}}C_\varphi(B\ovt M_i)$ and hence $\widetilde{A}^c\preceq_{\widetilde{N}}C_\varphi(B\ovt M_i)$ for some $i$. For simplicity we assume $i=1$. Using \cite[Lemma 4.8]{HI15} and since the inclusion $Q_0 \subset \widetilde{A}^c$ is with expectation, it holds that $Q_0\preceq_{\widetilde{N}}C_\varphi(B\ovt M_1)$. In this setting, we consider the following two cases:
	$$\text{(ii-a) }Q_0 \preceq_{\widetilde{N}} \widetilde{B}; \quad \text{or} \quad \text{(ii-b) }Q_0 \not\preceq_{\widetilde{N}}\widetilde{B} .$$

\noindent
\textbf{Assume that (ii-a)} $Q_0 \preceq_{\widetilde{N}}\widetilde{B}$. Then  \cite[Proposition 2.10]{BHR12} implies $Q_0 \preceq_N B$. Using $Q_0'\cap Q=\C$ and applying Lemma \ref{relative commutant} two times, we indeed get $Q\preceq_N B$.  
Applying \cite[Lemma 4.13]{HI15} and since $B,Q$ are type III factors, we can take a partial isometry $v$ such that $qBq=vpQpv^*\ovt L$, where $q=vv^*\in B'\cap N=M$, $p=v^*v\in Q'\cap N=P$, and $L$ is a factor. Since $P$ is a type III factor, $p$ is equivalent to $1$, so we may assume $v^*v=1$. Since $vQv^* \subset Bq$ and $q\in B'\cap N$, we get that $vQv^*$ is injective relative to $B$ inside $N$. 
Take a conditional expectation $E\colon q\langle N,B \rangle q \to v Qv^*$ which is faithful and normal on $qNq$, and consider the composition map $\Ad v^*\circ E\circ \Ad v\colon \langle N,B \rangle  \to Q$. It is easy to show that it is a conditional expectation $\langle N,B \rangle  \to Q$ which is faithful and normal on $N$. 
We obtain $Q\lessdot_N B$ which is the desired condition.

\vspace{0.5em}
\noindent
\textbf{Assume that (ii-b)} $Q_0 \not\preceq_{\widetilde{N}}\widetilde{B}$ and we will deduce a contradiction. 
Combined with the assumption $Q_0\preceq_{\widetilde{N}}C_\varphi(B\ovt M_1)$ and using (the proof of) \cite[Lemma 2.6]{HU15b}, there are $p\in Q_0$, $q\in C_\varphi(B\ovt M_1)$, $\theta\colon pQ_0 p\to qC_\varphi(B\ovt M_1)q$, $v\in \widetilde{N}$ such that they witness $Q_0\preceq_{\widetilde{N}}C_\varphi(B\ovt M_1)$ and that $\theta(pQ_0p)\not\preceq_{\widetilde{N}}\widetilde{B}$. 
Using the proof of \cite[Theorem 4.3(1)$\Rightarrow$(2-a)]{HI15}, up to replacing $q$ with a slightly smaller projection, we may assume $\Tr(q)<\infty$. 
Observe that the condition $\theta(pQ_0p)\not\preceq_{\widetilde{N}}\widetilde{B}$ implies $\theta(pQ_0p)\not\preceq_{C_{\varphi}(B\ovt M_1)}\widetilde{B}$ and hence by \cite[Theorem 2.4]{CH08}, it holds that $\theta(pQ_0p)'\cap q\widetilde{N}q \subset qC_\varphi(B\ovt M_1)q$. Since $vv^*\in \theta(pQ_0p)'\cap q\widetilde{N}q \subset qC_\varphi(B\ovt M_1)q$, up to replacing $q$ with $vv^*$ (and $\theta$ with $\theta(\, \cdot\, )vv^*$), we may assume $q=vv^*$. 
Observe that $\widetilde{P}$ is a II$_\infty$ factor with the trace $\Tr$, and $\Tr$ is semifinite on the diffuse subalgebra $\widetilde{P}_0:=\Pi_{\varphi,\psi}(C_\psi(P_0))$. So any projection in $\widetilde{P}$ is equivalent to a projection in $\widetilde{P}_0$. 
Since $v^*v\in (pQ_0p)'\cap p\widetilde{N}p = \widetilde{P}p$, it is equivalent to a projection in $\widetilde{P}_0p$. Up to replacing, we may assume that $v^*v\in \widetilde{P}_0p$. 
Summarizing we are in the situation that $vv^*=q \in C_\varphi(B\ovt M_1)$, $v^*v\in \widetilde{P}_0p$, together with the inclusion 
	$$ vQ_0 v^* = \theta(pQ_0p) \subset qC_\varphi(B\ovt M_1) q.$$
The assumption $Q_0\not\preceq_{\widetilde{N}}\widetilde{B}$ and \cite[Remark 4.2(2)]{HI15} imply $v^*vQ_0v^*v\not\preceq_{\widetilde{N}}\widetilde{B}$. This means $vQ_0v^*\not\preceq_{\widetilde{N}}\widetilde{B}$ and so $vQ_0v^*\not\preceq_{C_\varphi(B\ovt M_1)}\widetilde{B}$, and therefore \cite[Theorem 2.4]{CH08} implies 
	$$ v\widetilde{P}v^* = v (Q_0'\cap \widetilde{N}) v^* = (vQ_0v^*)'\cap  q\widetilde{N} q \subset qC_\varphi(B\ovt M_1) q. $$
Recall that we first assumed $A\not\preceq_N B$. By (the proof of) \cite[Proposition 2.10]{BHR12}, this implies $A\not\preceq_{\widetilde{N}}\widetilde{B}$. 
Since $A\subset \widetilde{P}_0$, we have $\widetilde{P}_0\not\preceq_{\widetilde{N}}\widetilde{B}$ by \cite[Lemma 4.8]{HI15}. Since $v^*v \in \widetilde{P}_0p$, by \cite[Remark 4.2(2)]{HI15}, we have $v^*v\widetilde{P}_0v^*v\not \preceq_{\widetilde{N}}\widetilde{B}$, which is equivalent to $v\widetilde{P}_0v^*\not \preceq_{\widetilde{N}}\widetilde{B}$.
This implies $v\widetilde{P}_0v^*\not \preceq_{C_\varphi(B\ovt M_1)}\widetilde{B}$ and therefore \cite[Theorem 2.4]{CH08} again can be applied, so that
	$$ v \widetilde{Q} v^* = v(P_0'\cap \widetilde{N}) v^* = (vP_0v^*)' \cap q\widetilde{N}q \subset qC_\varphi (B\ovt M_1)q.$$
In summary we obtain 
	$$ q\widetilde{N}q = v \widetilde{N} v^* = v (\widetilde{P}\vee \widetilde{Q}) v^*\subset qC_\varphi(B\ovt M_1)q.$$
This implies $ \widetilde{N} \preceq_{\widetilde{N}}C_\varphi(B\ovt M_1)$ and hence $C_\varphi(M_2)\preceq_{\widetilde{N}}C_\varphi(B\ovt M_1)$ by \cite[Lemma 4.8]{HI15}. Let $C \subset M_2$ be any diffuse abelian von Neumann algebra with expectation and let $\omega$ be a faithful normal state on $N$ such that $\omega\circ E_{M_2}=\omega$ and $C\subset (M_2)_\omega$. We have that $\Pi_{\varphi,\omega}(C_\omega(M_2))\preceq_{\widetilde{N}}C_\varphi(B\ovt M_1)$ and hence $\Pi_{\varphi,\omega}(C_\omega(C))\preceq_{\widetilde{N}}C_\varphi(B\ovt M_1)$ by \cite[Lemma 4.8]{HI15}. We apply \cite[Proposition 2.10]{BHR12} and get $C\preceq_{N}B\ovt M_1$. Lemma \ref{tensor lemma} then implies $C\preceq_M M_1$. Since $C\subset M_2$ is diffuse, which is equivalent to $C\not\preceq_{M_2}\C$, we obtain $C\not\preceq_{M}M_1$ by \cite[Lemma 2.7]{HU15b}, that is a contradiction. 
\end{proof}

\section{\bf Proof of Theorem \ref{ThmA}}\label{Proof of Theorem A}

\begin{proof}[Proof of Theorem \ref{ThmA}]
	Fix faithful normal states $\varphi_0$ and $\psi_0$ on $M_0$ and $N_0$ respectively. As in previous sections, we use the following notation: 
\begin{align*}
	(M,\varphi):= \ovt_{m\in \{0\}\cup X}(M_m,\varphi_m),\quad (N,\psi):=\ovt_{n\in \{0\}\cup Y}(N_n,\psi_n); \\ 
	M_{\mathcal{F}}:=\ovt_{n\in \mathcal{F}} M_n \subset M, \quad M_{\mathcal{F}}^c:=\ovt_{n\in \mathcal{F}^c} M_n \subset M, \quad \text{for all }\mathcal{F}\subset \{0\}\cup X.
\end{align*}
We use similar notations for $N_n$, such as $N_{\mathcal{F}}$ for $\mathcal{F}\subset \{0\} \cup Y$. We identify $ M = N$ for simplicity. 
We first assume all $M_m$ and $N_n$ are factors only and prove the following claim. The assumptions in the claim is more general than the one in the theorem, but this generality is necessary for the proof. 
\begin{claim}
	Assume that for all $m\in X$ there are projection $p_m\in M_m$ and $p_m' \in M_m^c$ such that $p_mM_m p_m p_m'= P_m \ovt R_m$, where $P_m$ is a factor in the class $\mathcal P$ and $R_m$ is an amenable factor. Assume that all $N_n$ are non-amenable factors. 
Then there is an injective map $\sigma\colon Y \to X$ such that $P_{\sigma(j)} \preceq_{M} N_j$ for all $j\in Y$, where $P_m$ is regarded as a unital subalgebra of $p_mp_m'Mp_mp_m'$ for all $m\in X$.
\end{claim}
\begin{proof}[Proof of Claim.]
Since $M_m$ is a factor, the map $M_m \ni x \mapsto xp_m' \in M_mp_m'$ is  isomorphic. So there is a decomposition $p_mM_mp_m = \widetilde{P}_m \ovt \widetilde{R}_m$ such that $\widetilde{P}_m p_m' = P_m$ and  $\widetilde{R}_mp_m' = R_m$ for all $m\in X$. 
For any $m\in X$ and $j \in Y$, we have that $\widetilde{P}_{m} \preceq_{M} N_j$ (where $\widetilde{P}_m \subset p_mMp_m$ is a unital subalgebra) if and only if $P_m = \widetilde{P}_{m}p_m' \preceq_{M} N_j$. To see this, use \cite[Remark 4.2(2)]{HI15} for the `if' direction, and \cite[Lemma 4.12(2)]{HI15} for the `only if' direction. 
Hence without loss of generality, we may assume $p_m'=1$ for all $m\in X$.

We next show that we may assume $p_m=1$ for all $m$. 
To see this, observe that if $p_m$ is equivalent to $1$ in $M_m$, then using a partial isometry $v_m \in M_m$ with $v_mv_m^*=1$ and $v_m^*v_m =p$, we have a decomposition $M_m = v_m P_mv_m^*\ovt v_mR_mv_m^*$. Then for any $j\in Y$, we have that $v_m P_mv_m^* \preceq_M N_j$ if and only if $P_m\preceq_M N_j$ (this holds by definition). 
If $p_m$ is a finite projection, it is easy to find a decomposition 
$M_m=\widetilde{P}_m \ovt R_m$ such that $\widetilde{P}_m$ is stably isomorphic to $P_m$ (hence in the class $\mathcal P$ by Lemma \ref{stably isom for class P}), and that there are projections $q_m\in P_m$ and $r_m \in \widetilde{P}_m$ satisfying $r_m\widetilde{P}_mr_m = q_mP_mq_m$. 
Then for any $j\in Y$, we have that $\widetilde{P}_m \preceq_M N_j$ if and only if $r_m\widetilde{P}_mr_m = q_mP_mq_m \preceq_M N_j$ by \cite[Remark 4.2(2) and 4.2(4)]{HI15}, and hence if and only if $P_m \preceq_M N_j$. 
Thus combined these observations together, to show this claim, we may assume $p_m=1$ for all $m$.

	Now we assume $p_m=p_m'=1$ for all $m\in X$. Fix $j\in Y$ and we will find $i\in X$ such that $P_i \preceq_{M} {N_j}$. For this, suppose by contradiction that $P_i \not\preceq_{M} {N_j}$ for all $i \in X$. By Lemma \ref{relative commutant}, this is equivalent to ${N}_j^c \not\preceq_{M} P_i' \cap M$ for all $i$, and by Corollary \ref{type of infinite tensor}, 
the factor $P_i' \cap M$ has a large centralizer (note that all $M_m=P_m\ovt R_m$ have large centralizers since so do all $P_m$ and $R_m$). 
Since $P_i$ is in the class $\mathcal P$, we have ${N}_j \lessdot_{M} P_i'\cap M$ for all $i\in X$. 
Observe that $P_i'\cap M = R_i\ovt M_i^c $ and $R_i \ovt M_i^c \lessdot_M M_i^c$ (use Corollary \ref{amenability for tensor product} as in the proof of Lemma \ref{stably isom for class P}).  Corollary \ref{transitivity} implies ${N}_j \lessdot_{M} M_i^c$ for all $i\in X$.  
Applying Lemma \ref{intersection lemma}, we can take intersections of $M_i^c$ for finitely many $i\in X$, that is, we have ${N}_j \lessdot_{M} M_{\mathcal{F}}^c$ for all finite subsets $\mathcal{F} \subset X$. We then apply Lemma \ref{infinite tensor lemma finite case} and get ${N}_j \lessdot_{M} M_{0}$. Since $ M_{0}$ is amenable, we conclude that ${N}_j$ is amenable which contradicts our assumption. 
Thus we have proved that for any $j\in Y$, there is $i \in X$ such that $P_i \preceq_{M} {N_j}$. We can then define a map $\sigma\colon Y \to X$ such that $P_{\sigma(j)} \preceq_{M} N_j$ for all $j\in Y$. 

	We next show that $\sigma$ is injective. Assume that $\sigma(j)=\sigma(j')$. By \cite[Lemma 4.13]{HI15}, take a partial isomrtry $v \in M$ such that $ v P_{\sigma(j)} v^* \subset vv^* N_j vv^*$ with expectation and that $vv^*=q \, q'$ for projections $q\in N_j$ and $q' \in N_j^c$. Since $vv^* N_j vv^* \simeq qN_jq$, we can find a diffuse abelian subalgebra $A \subset qN_jq$ with expectation such that $Aq' \subset  v M_{\sigma(j)} v^*$. 
Since $v^*v$ can be  also written by projections in $P_{\sigma(j)}$ and $P_{\sigma(j)}'\cap M$, since $P_{\sigma(j)} \preceq_{M}N_{j'}$ and since $P_{\sigma(j)}$ and $P_{\sigma(j)}'\cap M$ are factors, it holds that $v^*vP_{\sigma(j)}v^*v \preceq_{M}N_{j'}$ by \cite[Remark 4.2(4) and 4.5]{HI15}. Then consider the inclusion $v^*Aq'v \subset v^*vP_{\sigma(j)}v^*v$ and apply \cite[Lemma 4.8]{HI15}, so that $v^*Aq'v \preceq_{M}N_{j'}$. We get $Aq' \preceq_{M}N_{j'}$ and hence $A \preceq_{M}N_{j'}$ by \cite[Remark 4.2(2)]{HI15}. 
This implies $j=j'$ by Lemma \ref{uniqueness lemma} (if $j\neq j'$, we have $A \preceq_{M}N_{j}^c$, a contradiction), and we obtain  injectivity of $\sigma$.
\end{proof}

Now we start the proof. Assume that all $M_m$ are in the class $\mathcal{P}$ and that all $N_n$ are non-amenable. By the claim above, we can find an injective map $\sigma\colon Y \to X$ such that $M_{\sigma(j)} \preceq_{M} N_j$ for all $j\in Y$. This finishes the proof of the first part of the theorem.

	Next we assume that $N_j$ is semiprime for all $j\in Y$ and prove the surjectivity of $\sigma$. For each $j\in Y$, using \cite[Lemma 4.13]{HI15} and semiprimeness of $N_j$, there is a partial isometry $v_j \in M$ such that $v_jM_{\sigma(j)} v_j^*\ovt R_j = v_jv_j^*N_jv_jv_j^*$, where $R_j$ is an amenable factor. 
Observe that $v_jM_{\sigma(j)} v_j^*$ is in the class $\mathcal{P}$ for all $j\in Y$ by Lemma \ref{stably isom for class P}. 
By replacing the roles of $(M_m)_{m\in X}$, $(N_n)_{n\in Y}$ and $(P_m)_{m\in X}$ with $(N_n)_{n\in Y}$, $(M_m)_{m\in X}$ and $(v_{n}M_{\sigma(n)} v_{n}^*)_{n\in Y}$ respectively, we can apply the claim above and find an injective map $\tau\colon X\to Y$ such that $v_{\tau(i)}M_{\sigma(\tau(i))} v_{\tau(i)}^* \preceq_M M_i$ for all $i\in X$. As in the last part of the proof of the claim above, this implies $M_{\sigma(\tau(i))}  \preceq_M M_i$ for all $i\in X$. Then Lemma \ref{uniqueness lemma} implies $\sigma(\tau(i))=i$ for all $i\in X$, hence $\sigma$ is surjective.

	Finally, since the above $v_n$ is given as $v_nv_n^* = q_n q_n'$ and $v_n^*v_n = p_n p_n'$, where $p_n\in M_{\sigma(n)}$, $p_n' \in M_{\sigma(n)}^c$, $q_n\in N_n$, and $q_n'\in N_n^c$ are projections, there is an isomorphism
	$$p_nM_{\sigma(n)}p_n \ovt R_n \simeq q_nN_nq_n \quad \text{for all }n\in Y.$$
This finishes the proof.
\end{proof}

\begin{proof}[Proof of Corollary \ref{CorB}]
	Since all $N_j$ are prime, amenable factors $R_j$ in the last statement of Theorem \ref{ThmA} become type I factors. Hence if tensor product factors are isomorphic, then each tensor component is stably isomorphic. 

Conversely assume that each tensor component is stably isomorphic. For simplicity we assume that $M_n\ovt \B(\ell^2) = N_n \ovt \B(\ell^2)$ for all $n\in X=Y$. If $M_n$ and $N_n$ are properly infinite, then we have $M_n=N_n$, so we take any faithful normal state $\varphi_n$ on $M_n$ and $\psi_n$ on $N_n$ such that $\varphi_n$ and $\psi_n$ coincide via $M_n=N_n$. 
If $M_n$ is finite and $N_n$ is properly infinite, then we have $M_n\ovt \B(\ell^2)=N_n$. Take any product state $\varphi_n\ovt\omega$ on $M_n\ovt \B(\ell^2)$ and define $\psi_n$ on $N_n$ using $M_n\ovt \B(\ell^2)=N_n$. Define similarly if $M_n$ is properly infinite and $N_n$ is finite. 
Finally if both $M_n$ and $N_n$ are finite, we have $M_n\ovt \M_k(\C) = q_nN_nq_n \ovt \M_l(\C)$ for a nonzero projection $q_n\in N_n$ and $k,l\in \N$. By choosing appropriate $k,l\in \N$, we may assume that the trace value of $q_n$ is sufficiently close to 1. Define $\varphi_n$ and $\psi_n$ as traces on $M_n$ and $q_nN_nq_n$ respectively.
Summary we have the following isomorphism
	$$ \ovt_{n\in X}(M_n, \varphi_n) \ovt M_0= \ovt_{n\in Y}(q_n N_n q_n,\psi_n) \ovt N_0,  $$
where $M_0$ and $N_0$ are some Araki--Woods factors and $q_n\in N_n$ are projections (which is $1_{N_n}$ unless both $M_n$ and $N_n$ are finite). To consider the effect of $q_n$, for simplicity we assume that all $M_n$ and $N_n$ are $\rm II_1$ factors. Let $\tau_n$ be the trace for $N_n$. Observe that since we can control the value $\tau_n(q_n)$ for all $n$, we may assume that the element $q:=q_1\otimes q_2 \otimes q_3 \otimes \cdots$ defines a nonzero projection in $\ovt_{n\in Y}(N_n, \tau_n)$. Hence with a suitable choice of $(q_n)_n$, it is not hard to see that 
	$$  q\left(\ovt_{n\in Y}( N_n ,\tau_n) \right) q \simeq \ovt_{n\in Y}(q_n N_n q_n,\psi_n) .$$
Thus we obtain the desired stable isomorphism.
\end{proof}

\appendix 
\section{\bf Relative amenability for bimodules}\label{Relative amenability for bimodules}

	In this Appendix, we define and investigate relative amenability for bimodules. All of our studies are based on the work of Connes \cite{Co75} on amenability and the one of Anantharaman-Delaroche \cite{AD93} on co-amenability. 
Although most of our results here are straightforward generalizations, we give detailed proofs for the reader's convenience.

	Throughout the appendix, we use the following notation. 
For any von Neumann algebras $M$ and $B$, an \textit{$M$-$B$-module} $H={}_MH_{B}$ is a Hilbert space equipped with faithful normal unital $\ast$-homomorphisms $\pi_H\colon M \to \B(H)$ and $\theta_H \colon B^{\rm op}\to \B(H)$ such that $\pi_H(M)$ and $\theta_H(B^{\rm op})$ commute. 
All opposite items are denoted with circles, such as $B^\op=B^{\rm op}$. 
The conjugate module $\overline{H}$ is the conjugate Hilbert space of $H$ equipped with the $B$-$M$-module structure given by 
	$$\pi_{\overline{H}}(b)\theta_{\overline{H}}(x^\op)\overline{\xi}:= \overline{\pi(x^*)\theta((b^*)^\op)\xi}, \quad x\in M, \ b\in B, \ \xi \in H. $$
The set of all $B^\op$-module maps on $H$ will be denoted by $$\mathcal{L}_B(H_B):=\theta_H(B^\op)'\cap \B(H).$$
We always have $\pi_H(M) \subset \mathcal{L}_B(H_B)$.  
We denote by $\nu_H$ the $\ast$-homomorphism from the algebraic $\ast$-algebra generated by $M$ and $B^\op$ (i.e.\ the algebra $M\otimes_{\rm alg}B^\op$ with involution $(x\otimes b^\op)^*=x^* \otimes (b^*)^\op$) into the C$^\ast$-algebra generated by $\pi_H(M)$ and $\theta_H(B^\op)$. 
For $M$-$B$-modules $H$ and $K$, we will write $H \prec K$ if we have a \textit{weak containment}, that is, representations $\nu_H$ and $\nu_K$ satisfy $\|\nu_H(x)\|_\infty\leq \|\nu_K(x)\|_\infty$ for all $x\in M \ota B^\op$.

	Let $B \subset M$ be von Neumann algebras with operator valued weight, $p\in M$ a projection, and $A\subset pMp$ a von Neumann subalgebra with expectation. Consider $H=L^2(pM)$ as a $pMp$-$B$-module. Then we have 
	$$\mathcal{L}_B(H_B)=\theta_H(B^\op)'\cap \B(H) = p\langle M,B\rangle p$$
and the $B$-$pMp$-module $\overline{L^2(pM)}$ is canonically identified with the standard $B$-$pMp$-module $L^2(Mp)$, via the modular conjugation $J$ of $L^2(M)$: $\overline{L^2(M)}\ni\overline{\xi}\mapsto J\xi\in L^2(M)$.  
From these points of view, the study on bimodules in this appendix will be used in the study of relative amenability in Section \ref{Relative amenability for subalgebras}.

	The following definition is a generalization of \cite{PV11}, in which they treat only  finite von Neumann algebras. We introduce two notions of relative amenability which are equivalent for finite von Neumann algebras.
\begin{Def}\label{relative amenable def2}\upshape
	Let $B$ and $M$ be von Neumann algebras, $A\subset M$ a von Neumann subalgebra, and $H={}_MH_B$ an $M$-$B$-module. 
\begin{itemize}
	\item[$\rm (1)$] We say that $H={}_MH_B$ is \textit{left $A$-semidiscrete} if we have a weak containment
	$${}_M L^2(M)_A \prec {}_M H\otimes_B \overline{H}_A ,$$
where $\otimes_B$ is the Connes' relative tensor product (e.g.\ \cite[Chapter IX. \S 3]{Ta01}).
	\item[$\rm (2)$] Assume that $A\subset M$ is with expectation $E_A$. We say that ${}_MH_B$ is \textit{left $(A,E_A)$-injective} if there exists a conditional expectation 
	$$E\colon \mathcal{L}_B(H_B) \to \pi_H(A)\simeq A \quad \text{ such that  } \quad E(\pi_H(x))=E_A(x) \quad \text{for all } x\in M.$$
\end{itemize}
\end{Def}

Before starting our work on the relative amenability, we prepare several lemmas.

\begin{Lem}\label{lemma for reduction}
	Let $M,N$ be von Neumann algebras and ${}_MH_N$, ${}_MK_N$ be $M$-$N$-bimodules. Let $p\in M$, $q\in N$ be any projections such that central supports of $p$ in $M$ and $q$ in $N$ are $1_M$ and $1_N$ respectively. 
\begin{itemize}
	\item[$(1)$] There are canonical identifications 
	$${}_M L^2(Mp)\otimes_{pMp}(pH)_{N}\simeq{}_MH_N , \quad {}_M (Hq) \otimes_{qNq}L^2(qN)_N \simeq {}_MH_N.$$
	\item[$\rm(2)$] We have that ${}_MH_N \prec {}_MK_N$ if and only if ${}_{pMp}(pHq)_{qNq} \prec {}_{pMp}(pKq)_{qNq}$.
\end{itemize}
\end{Lem}
\begin{proof}
We note that the left $M$-action on $L^2(Mp)$ is faithful if and only if the central support projection of $p$ in $M$ is $1_M$. 

\noindent
	(1) It is easy to see that ${}_ML^2(Mp)\otimes_{pMp}L^2(pM)_M \simeq {}_ML^2(M)_M$. We have 
	$${}_M L^2(Mp)\otimes_{pMp}(pH)_{N}\simeq{}_M L^2(Mp)\otimes_{pMp}L^2(pM)\otimes_M H_N \simeq {}_ML^2(M)\otimes_{M}H_{N}\simeq{}_MH_N .$$
The same argument works for ${}_M (Hq) \otimes_{qNq}L^2(qN)_N$.

\noindent
	(2) The only if direction is trivial. To see the if part, using \cite[Lemma 1.7]{AD93}, apply ${}_M L^2(Mp)\otimes_{pMp}$ from the left and $\otimes_{qNq}L^2(qN)$ from the right side.
\end{proof}

\begin{Lem}\label{The Trick}
	Let $M,N$ be von Neumann algebras and ${}_MH_N$, ${}_MK_N$ be $M$-$N$-bimodules. If ${}_MH_N\prec {}_MK_N$, then there is a ucp map  $\Psi\colon\mathcal{L}_N(K_N)\to \mathcal{L}_N(H_N)$ such that $\Psi(\pi_K(x))=\pi_H(x)$ for all $x\in M$.
\end{Lem}
\begin{proof}
	Let $\nu$ be the bounded $\ast$-homomorphism for ${}_MH_N\prec {}_MK_N$, namely, it sends $\nu_K(x)$ into $\nu_H(x)$. 
The map $\nu$ is naturally defined on $\mathrm{C}^*\{ \pi_K(M), \theta_K(N^\op )\}$ and,  by Arveson's extension theorem, we extend it on $\mathrm{C}^*\{ \mathcal{L}_N(K_N), \theta_K(N^\op ) \}\subset \B(K)$ as a u.c.p.\ map into $\B(H)$. We denote by $\widetilde{\nu}$ this u.c.p.\ extension and then define $\Psi\colon \mathcal{L}_N(K_N) \to \B(H)$ by $\Psi(T):=\widetilde{\nu}(T)$. 
Obviously $\Psi(\pi_K(x))=\pi_H(x)$ for $x\in M$. We have to show that $\mathrm{Im}\Psi\subset \mathcal{L}_N(H_N)$, which means $\mathrm{Im}\Psi$ commutes with $\theta_H(N^\op)$. 
For any $u\in \mathcal{U}(N)$ and $T\in \mathcal{L}_N(K_N)$, since $\theta_K(u^\op)$ is contained in the multiplicative domain of $\nu$ (e.g.\ \cite[Proposition 1.5.7]{BO08}), we have 
	$$\Psi(T)\theta_H(u^\op)=\widetilde{\nu}(T)\nu(\theta_K(u^\op))
=\widetilde{\nu}(T\theta_K(u^\op))=\widetilde{\nu}(\theta_K(u^\op)T)=\theta_H(u^\op)\Psi(T).$$ 
Hence $\Psi(T)$ commutes with $\theta_H(u^\op)$ for all $u\in \mathcal{U}(N)$, and $\Psi$ is a desired ucp map.
\end{proof}

\begin{Lem}\label{GNS lemma}
	Let $M,N$ be von Neumann algebras and ${}_MH_N$, ${}_MK_N$ be $M$-$N$-bimodules. Assume that there is a cyclic vector $\xi\in {}_MH_N$, that is, $\pi_H(M)\theta_H(N^\op)\xi \subset H$ is dense. 
Then ${}_MH_N\prec {}_MK_N$ if and only if the linear functional
	$$\B(K)\supset\text{ \rm $\ast$-alg} \{ \pi_K(M),\theta_K(N^\op)\} \ni \pi_K(x)\theta_K(y^\op) \mapsto \langle \pi_H(x)\theta_H(y^\op) \xi, \xi \rangle_H \in \C$$
is bounded (with respect to the norm in $\B(K)$).
\end{Lem}
\begin{proof}
	The `only if' part is trivial. For the converse, observe that the given linear functional is positive on the $\ast$-algebra, so it can be extended on $\mathrm{C}^*\{ \pi_K(M),\theta_K(N^\op)\}$ as a positive linear functional. 
Then since $\xi$ is cyclic, the Hilbert space of the GNS representation of this functional is identified as $H$. In particular the GNS representation is identified with the $\ast$-homomorphism $\pi_K(x)\theta_K(y^\op) \mapsto \pi_H(x)\theta_H(y^\op)$, and hence it is bounded.
\end{proof}

\subsection*{\bf Characterizations of left injectivity/semidiscreteness}

	We start our work with proving well-known characterizations of relative amenability, which are generalizations of a part of \cite[Theorem 5.1]{Co75} and \cite[Section 3]{AD93}. For finite von Neumann algebras, they are proved in \cite[Theorem 2.1]{OP07} and \cite[Proposition 2.4]{PV11}. 

For this we prepare a lemma. We note that the condition (2) in this lemma is more general than the relative injectivity, which corresponds to the case $N=\mathcal{L}_B(H_B)$.
 
\begin{Lem}\label{relative amenable theorem1}
	Let $A\subset M$ be $\sigma$-finite von Neumann algebras with expectation $E_A$, and let $N$ be a von Neumann algebra containing $M$. Assume that $A$ is finite. We fix a trace $\tau_A$ on $A$ and put $\psi:=\tau_A\circ E_A$. Then the following conditions are equivalent.
	\begin{itemize}
		\item[$\rm(1)$] There is an $A$-central state $\widetilde{\psi}$ on $N$ such that $\widetilde{\psi}\mid_{M}=\psi$.
		\item[$\rm(2)$] There is a conditional expectation from $N$ onto $A$, which restricts to $E_A$ on $M$.
		\item[$\rm(3)$] There is a net $(\xi_i)_i$ of unit vectors in the positive cone of $L^2(N)$ such that 
$$\langle x \xi_i , \xi_i\rangle\rightarrow \psi(x) \quad \text{for all }x\in M \quad\text{and}\quad \|uJ_NuJ_N\xi_i-\xi_i\|_2\rightarrow 0 \quad \text{for all }u\in \mathcal{U}(A),$$ 
where $J_N$ is the modular conjugation for $L^2(N)$.
	\end{itemize}
\end{Lem}
\begin{proof}
The proof is almost identical to the one of \cite[Theorem 2.1]{OP07}. 	Hence we give a sketch of the proof.

\noindent
	(1) $\Rightarrow$ (2) Using the $A$-centrality, we have
$$|\widetilde{\psi}(ax)|\leq \|a\|_{1,\tau_A} \|x\|_\infty \quad \text{for all } a\in A, \ x\in N.$$
For any $x\in N$, define a functional $T_x\colon A\to \C$ by $T_x(a):=\widetilde{\psi}(ax)$. This is bounded on $L^1(A,\tau_A)$ and so there is a unique element $\Phi(x)\in A$ such that $\tau_A(a\Phi(x))=\widetilde{\psi}(ax)$ for all $a\in A$. This $\Phi$ is a desired conditional expectation.

\noindent
	(2) $\Rightarrow$ (1) Compose $\psi$ with the given conditional expectation.

\noindent
	(1) $\Rightarrow$ (3) Let $(\psi_i)_i$ be a net of normal states on $N$ converging to $\widetilde{\psi}$ weakly. This satisfies that for any $u\in\mathcal{U}(A)$, the net $(u\psi_iu^*- \psi_i)_i$ converges to zero weakly, where $u\psi_i u^*(x):=\psi_i(u^*xu)$ for $x\in N$. So by the Hahn--Banach separation theorem, one has that, up to replacing with convex combinations, the net $(u\psi_iu^*- \psi_i)_i$ converges to zero in the norm topology of $N_*$ for all $u\in \mathcal{U}(A)$. 
For each $i$, let $\xi_i$ be the unit vector in the positive cone of $L^2(N)$ such that the vector state of $\xi_i$ is $\psi_i$. Then the Powers--St$\rm \o$rmer inequality \cite[Theorem IX.1.2(iv)]{Ta01} shows that 
$$\|uJ_NuJ_N\xi_i-\xi_i\|_2^2 \leq \|u\psi_iu^*- \psi_i\|_{N_*} \rightarrow 0 \quad \text{for all } u\in \mathcal{U}(A).$$

\noindent
	(3) $\Rightarrow$ (1) Define a state on $N$ by $\widetilde{\psi}(x):=\mathrm{Lim}_i\langle x\xi_i , \xi_i \rangle$. 
\end{proof}

\begin{Thm}\label{relative amenable theorem2}
	Let $B,M,A,E_A$ and ${}_MH_B$ as in Definition \ref{relative amenable def2}(2) and consider the following conditions. 
	\begin{itemize}
		\item[$(1)$] The bimodule ${}_MH_B$ is left $(A,E_A)$-semidiscrete. 
		\item[$(2)$] There is a $B$-$A$-bimodule $K$ such that ${}_{M}L^2(M)_A \prec {}_{M} H\otimes_B K_A$.
		\item[$(3)$] There is a ucp map $\Psi\colon \mathcal{L}_B(H_B) \to \langle M, A \rangle $ such that $\Psi(\pi_H(x))=x$ for all $x\in M$.
		\item[$(4)$] The bimodule ${}_MH_B$ is left $(A,E_A)$-injective.
	\end{itemize}
Then we have $(1)$ $\Rightarrow$ $(2)$ $\Rightarrow$ $(3)$ $\Rightarrow$ $(4)$. If we further assume $A$ is semifinite, then $(4)$ $\Rightarrow$ $(1)$ holds, so all conditions are equivalent.
\end{Thm}
\begin{proof}
	The implication (1) $\Rightarrow$ (2) is trivial and (2) $\Rightarrow$ (3) follows from Lemma \ref{The Trick}. 
To see (3) $\Rightarrow$ (4), apply the compression map by the Jones projection $e_A$ of $E_A$.

\noindent
(4) $\Rightarrow$ (1)
	Assume first that $A$ is finite with a faithful normal trace $\tau_A$. Put $\psi:=\tau_A\circ E_A$. We apply Lemma \ref{relative amenable theorem1} below for $N:=\mathcal{L}_B(H_B)$, and get a net of unit vectors $(\xi_i)_i\in L^2(N)$ such that $\langle \pi_H(x)\xi_i , \xi_i\rangle\rightarrow \psi(x)$ for all $x\in M$ and $\|\pi_H(u)J\pi_H(u)J\xi_i-\xi_i\|_2\rightarrow 0$ for all $u\in \mathcal{U}(A)$, where $J$ is the modular conjugation for $L^2(N)$. 
Observe by \cite[PROPOSITION 3.1]{Sa81} that
	$$L^2(N)=L^2(\theta_H(B^\op )')\simeq H\otimes_B \overline{H}$$ 
as $N$-bimodules, and hence as $M$-bimodules. 
Note that the $M$-bimodule structure of $H\otimes_B \overline{H}$ here is given by the left action $\pi_H(x)\otimes_B 1 $ and the right one $1\otimes_B \theta_{\overline{H}}(y^\op) $ for all $x,y \in M$. 
We regard $(\xi_i)_i$ as vectors in $H\otimes_B \overline{H}$. Then the second condition on $(\xi_i)_i$ is translated as follows: for any $a\in \mathcal{U}(A)$
	$$\|(\pi_H(a)\otimes_B 1)\xi_i-(1\otimes_B \theta_{\overline{H}}(a^\op ))\xi_i\|_2=\|(\pi_H(a)\otimes_B \theta_{\overline{H}}((a^\op )^*))\xi_i-\xi_i\|_2\rightarrow 0.$$
Using the first condition on $(\xi_i)_i$ together, we obtain
	$$\langle (\pi_H(x)\otimes_B \theta_{\overline{H}}(a^\op))\xi_i , \xi_i \rangle\rightarrow \psi(xa)= \langle x a^\op\xi_\psi , \xi_\psi \rangle_{\psi} \qquad (x\in M, a\in A),$$
where $\xi_\psi \in L^2(M,\psi)$ is the canonical cyclic vector. 
In particular the linear functional $\pi_H(x)\otimes_B \theta_{\overline{H}}(a^\op)\mapsto \langle x a^\op\xi_\psi , \xi_\psi \rangle_{\psi}$ is bounded. So by Lemma \ref{GNS lemma} we get ${}_ML^2(M,\psi)_A\prec {}_MH\otimes_B \overline{H}_A$, which is our conclusion.

	We next show the general case. So assume that $A$ is semifinite. 
Let $\nu$ be the algebraic $\ast$-homomorphism for the weak containment ${}_ML^2(M)_A \prec {}_M H\otimes_B \overline{H}_A$. We will show that $\nu$ is bounded. For this we fix $x\in \text{\rm $\ast$-alg}\{ \pi_H(M)\otimes_B 1, 1\otimes_B\theta_{\overline{H}}(A^\op) \}$ and we will show $\|\nu(x)\|_\infty \leq \|x\|_\infty$.

	Let $p\in A$ be a finite and $\sigma$-finite projection. Then the assumption (4) implies that ${}_{pMp}(pH)_B$ is left $(pAp,E_{pAp})$-injective. Since $pAp$ has a faithful normal trace, by the result we already proved, ${}_{pMp}(pH)_B$ is left semidiscrete. 
Put $p_H:=\pi_H(p)\theta_H(p^\op )$ and $\tilde{p}:=pp^\op$, and observe that left semidiscreteness of ${}_{pMp}(pH)_B$ implies 
	$$\|\tilde{p} \nu(x)\tilde{p}\|_\infty = \|\nu(p_H xp_H)\|_\infty\leq \|p_H xp_H\|_\infty.$$ 
Next take a net $(p_i)_i$ of finite and $\sigma$-finite projections in $A$ which converges to $1_A$ strongly. Taking the supremum of such $p_i$, we obtain 
	$$\|\nu(x)\|_\infty= \sup_{i}\|\tilde{p_i} \nu(x)\tilde{p_i}\|_\infty\leq  \sup_i \|{p}_{i,H} x{p}_{i,H}\|_\infty =\|x\|_\infty.$$
Here we used the following fact: for any projections $p_i\in \B(K)$ converging to 1 strongly on any Hilbert space $K$, we have $\|X\|_\infty=\sup_i\|p_iXp_i\|_\infty$ for any $X\in\B(K)$. 
Thus we get the boundedness of $\nu$ and this is the desired condition.
\end{proof}

\subsection*{\bf Continuous core approach}

	We next study relative amenability using continuous cores. The use of the continuous core is natural in our context because, as observed in Theorem \ref{relative amenable theorem2}, the tracial condition is crucial to obtain the equivalence of semidiscreteness and injectivity.

	We fix the following setting. Let $B,M$ be von Neumann algebras and $A\subset M$ a von Neumann subalgebra with expectation $E_A$. 
Let $\psi_A$ be a faithful normal semifinite weight on $A$. Put  $\psi:=\psi_A\circ E_A$ and recall that continuous cores have an embedding $C_\psi(A)\subset C_\psi(M)$. 
Let $H={}_MH_B$ be an $M$-$B$-bimodule and define a $C_\psi(M)$-$B$-bimodule ${}_MK_B$ by $K:= {}_{C_\psi(M)} L^2(C_\psi(M))\otimes_M H_B$. 
Note that, under the isomorphism 
\begin{align*}
	{}_{C_\psi(M)} L^2(C_\psi(M))\otimes_M H_B 
	&\simeq {}_{C_\psi(M)} \left(L^2(\R) \otimes L^2(M)\right)\otimes_M \left(L^2(M)\otimes_M H\right)_B \\
	&\simeq {}_{C_\psi(M)} L^2(\R) \otimes L^2(M)\otimes_M H_B \\
	&\simeq {}_{C_\psi(M)} L^2(\R) \otimes H_B ,
\end{align*}
our actions are of the forms: for $x\in M$, $t\in \R$, $y\in B$,
	$$\pi_K(x):=\pi_{\sigma^\psi}(x), \quad \pi_K(\lambda_t):=\lambda_t\otimes 1, \quad \rho_K(y^{\rm op}):=1\otimes \theta_H(y^{\rm op}),$$
where $\pi_{\sigma^\psi}(x)$ is the usual representation in crossed products given by $(\pi_{\sigma^\psi}(x)\xi)(s)=\sigma_{-s}^\psi(x)\xi(s)$ for all $x\in M$ and $\xi \in L^2(\R)\otimes H = L^2(\R, H)$.

	The following theorem establishes an equivalence of semidiscreteness and injectivity, using the bimodule ${}_M K_B$. This equivalence will be used in the proof of the main theorem.

\begin{Thm}\label{injectivity by core}
	Keep the setting as above and consider the following conditions.
	\begin{itemize}
		\item[$(1)$] The bimodule ${}_MH_B$ is left $(A,E_A)$-injective. 
		\item[$(2)$] The bimodule ${}_{C_\psi(M)}K_B$ is left $C_\psi(A)$-semidiscrete, that is, 
	$${}_{C_\psi(M)}L^2(C_\psi(M))_{C_\psi(A)} \prec {}_{C_\psi(M)}K\otimes_B \overline{K}_{C_\psi(A)} .$$
		\item[$(3)$] The bimodule ${}_{M}K_B$ is left $A$-semidiscrete, that is, 
	$${}_{M}L^2(M)_{A} \prec {}_{M}K\otimes_B \overline{K}_{A} .$$
	\end{itemize}
Then we have $(1)\Rightarrow(2)\Rightarrow(3)$. 
If we further assume that there is an operator valued weight from $\mathcal{L}_B(H_B)$ to $\pi_H(M)$, then we have $(3)\Rightarrow(1)$, so all conditions are equivalent.
\end{Thm}
\begin{proof}
	(1) $\Rightarrow$ (2)
Let $e_A$ be the Jones projection for $E_A$. Observe that the compression map by $e_A$ defines a faithful normal conditional expectation from $C_\psi(M)$ onto $C_\psi(A)$, which restricts to $E_A$ on $M$. We denote this expectation again by $E_A$. We will show that ${}_{C_\psi(M)}K_B$ is left $(C_\psi(A),E_A)$-injective which is equivalent to (2) by Theorem \ref{relative amenable theorem2}. 
By assumption there is a conditional expectation $E\colon\mathcal{L}_B(H_B)\to A$ which restricts to $E_A$, where we omit $\pi_H$. 
We can then construct a conditional expectation  
	$$\widetilde{E}\colon \mathcal{L}_B(H_B) \ovt \B(L^2(\R))\to A\ovt \B(L^2(\R))$$ 
which restricts to $E_A \otimes \id$ on $M\ovt \B(L^2(\R))$. To see this, take finite rank projections $p_n\in \B(L^2(\R))$ which converges to 1 strongly. Then since $E\otimes \id$ is defined on $\mathcal{L}_B(H_B) \ovt p_n\B(L^2(\R))p_n$, we can define $\widetilde{E}$ as a cluster point of maps $x \mapsto (E\otimes \id )((1\otimes p_n)x(1\otimes p_n))$. One can directly check that $\widetilde{E}|_{M\ovt \B(L^2(\R))}$ is normal by applying any normal tensor states $\omega_1\otimes \omega_2$. Hence we get $\widetilde{E}|_{M\ovt \B(L^2(\R))}=E_A \otimes \id$.

	Observe that, by omitting $\pi_H$, $\mathcal{L}_B(H_B) \ovt \B(L^2(\R))$ and $A \ovt \B(L^2(\R))$ contain $C_\psi(M)$ and $C_\psi(A)$ respectively, and the restriction of $\widetilde{E}$ on $C_\psi(M)$ is $E_A$. 
Observe next that $A\ovt \B(L^2(\R))\simeq C_\psi(A) \rtimes \widehat{\R}$ by the Takesaki duality \cite[Theorem X.2.3]{Ta01}, so there is a conditional expectation from $A\ovt \B(L^2(\R))$ onto $C_\psi(A)$. 
By composing this expectation with $\widetilde{E}$, we get a conditional expectation from $\mathcal{L}_B(H_B) \ovt \B(L^2(\R))$ onto $C_\psi(A)$ which restricts to $E_A$ on $C_\psi(M)$. 
Finally observe that $\theta_K(B^\op)' = \mathcal{L}_B(H_B)\ovt \B(L^2(\R))$ and the inclusion $C_\psi(M)\subset \mathcal{L}_B(H_B)\ovt \B(L^2(\R))$ mentioned above coincides with the one of $\pi_K(C_\psi(M))\subset \theta_K(B^\op)'$. 
Hence we have constructed a conditional expectation from $\theta_K(B^\op)'$ onto $C_\psi(A)$ which restricts to $E_A$ on $C_\psi(M)$. This is the desired condition.

\noindent
(2) $\Rightarrow$ (3) 
	By definition, we have ${}_{M}L^2(C_\psi(M))_{A} \prec {}_{M}K\otimes_B \overline{K}_{A}$. We claim ${}_{M}L^2(M)_{M}\prec {}_{M}L^2(C_\psi(M))_{M}$, which obviously implies (3).

	To see the claim, we have to show for any $x_i,y_i\in M$, $i=1,\ldots,n$, 
	$$\left\|\sum_ix_iy_i^\op \right\|_\infty \leq \left\|\sum_i\pi_{\sigma^\psi}(x_i)(y_i^\op\otimes1)\right\|_\infty.$$
Since $\pi_{\sigma^\psi}(M)$ and $M^\op\otimes 1$ are contained in $\B(L^2(M))\ovt L^\infty(\R)\simeq L^\infty(\R,\B(L^2(M)))$, the right hand side in this inequality coincides with
	$$\text{\rm ess-}\sup_{t\in \R}\left\|\sum_i\sigma^\psi_t(x_i)y_i^\op\right\|_\infty.$$
Since the map $\R\ni t \mapsto \sum_i \sigma^\psi_t(x_i)y_i^\op$ is strongly continuous, the map $\R\ni t \mapsto \left\|\sum_i\sigma^\psi_t(x_i)y_i^\op\right\|_\infty$ is lower semi-continuous. Hence for any $\varepsilon>0$, there is $\delta>0$ such that 
	$$\left\|\sum_ix_iy_i^\op\right\|_\infty - \varepsilon \leq \left\|\sum_i\sigma^\psi_{t}(x_i)y_i^\op\right\|_\infty$$
for all $|t|<\delta$, and therefore 
	$$\left\|\sum_ix_iy_i^\op \right\|_\infty - \varepsilon\leq \left\|\sum_i\pi_{\sigma^\psi}(x_i)(y_i^\op\otimes1)\right\|_\infty.$$
Letting $\varepsilon \to 0$, the claim is proven.

\noindent
(3) $\Rightarrow$ (1) 
	Assume that there is an operator valued weight $E_M \colon \mathcal{L}_B(H_B) \to M$. Define a faithful normal semifinite weight on  $\mathcal{L}_B(H_B)$ by $\widehat{\psi}:=\psi\circ E_M$. It then holds that $\sigma_t^{\widehat{\psi}} |_M= \sigma_t^\psi$ and hence there is an inclusion $C_\psi(M) \subset C_{\widehat{\psi}}(\mathcal{L}_B(H_B))$. 
By assumption and Theorem \ref{relative amenable theorem2}, ${}_{M}K_B$ is left $A$-injective, so there is a conditional expectation 
	$$E\colon \mathcal{L}_B(K_B)= \mathcal{L}_B(H_B)\ovt \B(L^2(\R)) \to \pi_K(A)=\pi_{\sigma^\psi}(A)$$ 
which restricts to $E_A$ on $\pi_{\sigma^\psi}(M)$. Observe that $\mathcal{L}_B(H_B)\ovt \B(L^2(\R))$ contains $C_{\widehat{\psi}}(\mathcal{L}_B(H_B))$. By restriction, we have a conditional expectation from $\pi_{\sigma^{\widehat{\psi}}}(\mathcal{L}_B(H_B))$ onto $\pi_{\sigma^\psi}(A)$ which restricts to $E_A$. This means (1).
\end{proof}

In the case $A=M$, the following corollary is well known to experts but it is not explicitly written in \cite{AD93}. 
The corollary states that a conditional expectation can be approximated by normal ccp maps \textit{up to} Morita equivalence.

\begin{Cor}\label{approximation of expectation}
	Let $A\subset M \subset N$ be von Neumann algebras. Assume that there is a conditional expectation $E\colon N \to A$ which restricts to a faithful normal conditional expectation $E_A$ on $M$. 
Let $\psi_A$ be a faithful normal semifinite weight and put $\psi:=\psi_A\circ E_A$. 
Let $\pi_{\sigma^\psi}\colon M \to M\rtimes_\psi \R \subset N\ovt \B(L^2(\R))$ be the canonical embedding. 
\begin{itemize}
	\item[$\rm (1)$] There is a conditional expectation from $N\ovt \B(L^2(\R))$ onto $\pi_{\sigma^\psi}(A)$ which restricts to $E_A$ on $\pi_{\sigma^\psi}(M)$ and which is approximated by normal ccp maps from $N\ovt \B(L^2(\R))$ to $\pi_{\sigma^\psi}(A)$ in the point $\sigma$-weak topology.
	\item[$\rm (2)$] Assume further that there is an operator valued weight from $N$ to $M$. Then there is a conditional expectation from $N$ onto $A$ which restricts to $E_A$ on $M$ and which is approximated by normal ccp maps from $N$ to $A$ in the point $\sigma$-weak topology.
\end{itemize}
\end{Cor}
\begin{proof}
	(1) Fix $N\subset \B(H)$ and put $B:=(N')^\op$ and $H={}_MH_B$. By assumption, there is a conditional expectation $\mathcal{L}_B(H_B)\to A$ which restricts to $E_A$. So ${}_MH_B$ is left $(A,E_A)$-injective. 
By Theorem \ref{injectivity by core} (1)$\Rightarrow$(3), we have ${}_{M}L^2(M)_{A} \prec {}_{M}K\otimes_B \overline{K}_{A}$.
Since $K\otimes_B \overline{K}$ is the standard representation of $\theta_K(B^\op)'=N\ovt \B(L^2(\R))=:\widetilde{N}$ \cite[PROPOSITION 3.1]{Sa81}, we have
	$${}_{M}L^2(M)_{A} \prec {}_{M}K\otimes_B \overline{K}_{A} = {}_{M}L^2(\widetilde{N})_{A} = {}_{M}L^2(\widetilde{N})\otimes_{\widetilde{N}}L^2(\widetilde{N})_{A} .$$
This means ${}_{M}L^2(\widetilde{N})_{\widetilde{N}}$ is left $A$-semidiscrete. 
For any right $A$-module $L=L_A$, we denote by $X(L):=\mathrm{Hom}_{A}(L^2(A),L)$ the set of all bounded linear maps from $L^2(A)$ to $L$ which commute with right $A$-actions and  define an $A$-valued inner product by $\langle T ,S\rangle_{X(L)}:=T^*S \in (A^\op)'=A$. 
See \cite[Preliminaries]{AD93} for the relation of bimodules and W$^*$-Hilbert modules as well as for general theories of them. 
Then, the above weak containment is equivalent to  
	$$ {}_MX(L^2(M)) \prec {}_MX(L^2(\widetilde{N})) ,$$
where we are thinking them as Hilbert $A$-modules with left $M$-actions. Here $X(L^2(M))$ is identified as $M$ with the inner product $\langle x,y\rangle_{X(L^2(M))}=E_A(x^*y)$ for $x,y\in M$. By the weak containment, for the vector $1_M \in X(L^2(M))$, any $\sigma$-weak neighborhood $\mathcal V$ of 0, and any finite subset $\mathcal{E}\subset M$, there are vectors $\eta_i\in X(L^2(\widetilde{N}))$, $i=1,\ldots,n$, such that 
	$$ \langle x1_M , 1_M \rangle_{X(L^2(M))} - \sum_{i=1}^n\langle x\eta_i , \eta_i \rangle_{X(L^2(\widetilde{N}))} \in \mathcal V$$
for all $x\in \mathcal{E}$. 
We define a normal completely positive map from $\widetilde{N}$ into $A$  by 
	$$\varphi_{(\mathcal V, \mathcal E)}(T) := \sum_{i=1}^n\langle T\eta_i , \eta_i \rangle_{X(L^2(\widetilde{N}))}.$$
Observe that for any $x\in \mathcal{E}$, 
	$$ E_A(x) - \varphi_{(\mathcal V, \mathcal E)}(x) \in \mathcal V.$$
Hence letting $\mathcal{E}$ larger and $\mathcal V$ smaller, we have that $\varphi_{(\mathcal V, \mathcal E)}(x) \to E_A(x)$ $\sigma$-weakly for any fixed $x\in M$. By \cite[Lemma 1.6]{AD93}, regarding $\varphi_{(\mathcal V, \mathcal E)}$ as cp maps from $M$ into A, up to convex combinations and up to transforms $\varphi_{(\mathcal V, \mathcal E)} \mapsto b\varphi_{(\mathcal V, \mathcal E)} b^*$ for $b\in A$, we may assume that $\varphi_{(\mathcal V, \mathcal E)}(1)\leq E_A(1)=1$, hence ccp maps. 
Since the resulting ccp maps are still finite sums of  $M \ni x \mapsto \langle x\eta , \eta \rangle_{X(L^2(\widetilde{N}))}$ for $\eta\in X(L^2(\widetilde{N}))$, we can again regard $\varphi_{(\mathcal V, \mathcal E)}$ as normal cp maps from $\widetilde{N}$ to $A$, which are ccp by conditions $\varphi_{(\mathcal V, \mathcal E)}(1)\leq E_A(1)=1$. 
Finally any cluster point of $\varphi_{(\mathcal V, \mathcal E)}$ is a conditional expectation from $\widetilde{N}$ onto $A$ which restricts to $E_A$. Hence we can find a desired net of ccp maps as a subnet of $(\varphi_{(\mathcal V, \mathcal E)})_{(\mathcal V, \mathcal E)}$.

\noindent
(2) Take a faithful normal semifinite weight $\widehat{\psi}$ on $N$ such that $\sigma_t^{\widehat{\psi}} |_M= \sigma_t^\psi$, so that we have inclusions $C_\psi(M) \subset C_{\widehat{\psi}}(N)\subset N\ovt \B(L^2(\R))$. Then take a conditional expectation constructed in (1) and restrict it on $\pi_{\widehat{\psi}}(N)$. We get a conditional expectation from $\pi_{\widehat{\psi}}(N)$ onto $\pi_{\psi}(A)$ which restricts to $E_A$ on $\pi_\psi(M)$ and which is approximated by normal ccp maps. This is the conclusion.
\end{proof}

\bigskip

\small{

}
\end{document}